\documentclass[a4paper,10pt]{amsart}
\usepackage[utf8]{inputenc}
\usepackage{amssymb,amsmath,bbm}
\usepackage[all]{xy}

\newtheorem{theorem}{Theorem}[section]
\newtheorem{definition}[theorem]{Definition}
\newtheorem{example}[theorem]{Example}
\newtheorem{proposition}[theorem]{Proposition}
\newtheorem{lemma}[theorem]{Lemma}
\newtheorem{remark}[theorem]{Remark}

\newtheorem{corollary}[theorem]{Corollary}

\renewcommand{\AA}{\mathbb{A}}
\newcommand{\BB}{\mathcal{B}}
\newcommand{\CC}{\mathbb{C}}
\newcommand{\EE}{\mathcal{E}}
\newcommand{\FF}{\mathcal{F}}
\newcommand{\kk}{\Bbbk}
\newcommand{\KK}{\mathbb{K}}
\newcommand{\LL}{\mathbb{L}}
\newcommand{\NN}{\mathbb{N}}
\newcommand{\PP}{\mathbb{P}}
\newcommand{\QQ}{\mathbb{Q}}
\newcommand{\TT}{\mathbb{T}}
\newcommand{\ZZ}{\mathbb{Z}}
\newcommand{\Mst}{\mathfrak{M}}
\newcommand{\Msp}{\mathcal{M}}

\newcommand{\unit}{\mathbbm{1}}
\newcommand{\ICS}{\mathcal{IC}}
\newcommand{\DTS}{\mathcal{DT}}
\newcommand{\Part}{\mathcal{P}}
\newcommand{\Xst}{\mathfrak{X}}
\newcommand{\GG}{\mathbb{G}}

\DeclareMathOperator{\Hom}{Hom}
\DeclareMathOperator{\End}{End}
\DeclareMathOperator{\Ext}{Ext}
\DeclareMathOperator{\dgV}{dg-Vect}
\DeclareMathOperator{\Vect}{Vect}
\DeclareMathOperator{\rep}{-Rep}
\DeclareMathOperator{\Ka}{K}
\DeclareMathOperator{\Aut}{Aut}
\DeclareMathOperator{\codim}{codim}
\DeclareMathOperator{\Perv}{Perv}
\DeclareMathOperator{\MHM}{MHM}
\DeclareMathOperator{\rat}{rat}
\DeclareMathOperator{\IC}{IC}
\DeclareMathOperator{\DT}{DT}
\DeclareMathOperator{\Sym}{Sym}
\DeclareMathOperator{\Alt}{Alt}
\DeclareMathOperator{\Var}{Var}
\DeclareMathOperator{\Spec}{Spec}
\DeclareMathOperator{\Gl}{GL}
\DeclareMathOperator{\id}{id}
\DeclareMathOperator{\pr}{pr}
\DeclareMathOperator{\Quot}{Quot}
\DeclareMathOperator{\Gr}{Gr}
\DeclareMathOperator{\Char}{char}

\title[DT invariants vs.\ intersection cohomology of quiver moduli]{Donaldson--Thomas invariants versus intersection cohomology of quiver moduli}
\author{Sven Meinhardt \and Markus Reineke}

\begin{document}

\begin{abstract}
The main result of this paper is the statement that the Hodge theoretic Donaldson--Thomas invariant for a quiver with zero potential and a generic stability condition agrees with the compactly supported intersection cohomology of the closure of the stable locus inside the associated coarse moduli space of semistable quiver representations. In fact, we prove an even stronger result relating the Donaldson--Thomas ``function'' to the intersection complex. The proof of our main result relies on a relative version of the integrality conjecture in Donaldson--Thomas theory. This will be the topic of the second part of the paper, where the relative integrality conjecture will be proven in the motivic context.
\end{abstract}

\maketitle

\tableofcontents

\section{Introduction}

The theory of Donaldson--Thomas invariants  started around 2000 with the seminal work of R.\ Thomas \cite{Thomas1}. He associated  integers to moduli spaces in the absence of strictly semistable objects. Six years later D. Joyce \cite{JoyceI},\cite{JoyceCF},\cite{JoyceII},\cite{JoyceIII},\cite{JoyceMF},\cite{JoyceIV} and Y.\ Song \cite{JoyceDT} extended the theory, producing (possibly rational) numbers even in the presence of semistable objects which is the generic situation. Around the same time, M.\ Kontsevich and Y.\ Soibelman \cite{KS1},\cite{KS2},\cite{KS3} independently proposed a theory producing motives instead of simple numbers, also in the presence of semistable objects. The technical difficulties occurring in their approach disappear in the special situation of representations of quivers (with zero potential). This case has been  intensively studied by the second author in a series of papers \cite{Reineke2},\cite{Reineke3},\cite{Reineke4}. \\
Despite some computations of motivic or even numerical Donaldson--Thomas invariants for quivers with or without potential (see \cite{BBS},\cite{DaMe1},\cite{DaMe2},\cite{MMNS}), the true nature of Donaldson--Thomas invariants still remains mysterious.\\[1ex]
This paper is a first step to disclose the secret by showing that the Donaldson--Thomas invariants for quiver representations compute the compactly supported intersection cohomology of the closure of the stable locus inside the associated coarse moduli space of semistable representations. While trying to prove this result, the authors observed the importance of the integrality conjecture, which was the reason to extend the paper by a second part containing its proof. \\ 
We will actually prove an even stronger version by defining a Donaldson--Thomas function on the coarse moduli space $\Msp^{ss}$. Strictly speaking, this ``function'' is an element in a suitably extended Grothendieck group of mixed Hodge modules. The cohomology with compact support of that element is the usual Hodge theoretic Donaldson--Thomas invariant - a class in the Grothendieck group of mixed Hodge structures. Our main result is the following (we refer to the following sections for precise notation):
\begin{theorem}
 For a  generic stability condition the Donaldson--Thomas function is the class of the intersection complex $\ICS_{\overline{\Mst^{st}}}(\QQ)$ of the closure of the stable locus $\Msp^{st}$ inside the coarse moduli space $\Msp^{ss}$. In particular, by taking cohomology with compact support, we obtain for every dimension vector $d$ 
 \[ \DT_d=\begin{cases} \IC_c(\Msp^{ss}_d,\QQ)= \IC(\Msp^{ss}_d,\QQ)^\vee &\mbox{ if } \Msp^{st}_d\neq \emptyset, \\
          0 & \mbox{ otherwise}
         \end{cases} \]
in the Grothendieck ring of (polarizable) mixed Hodge structures.
\end{theorem}
As Donaldson--Thomas invariants for quiver representations can be computed with computer power quite effectively, this theorem provides a quick algorithm to determine intersection Hodge numbers. The previous algorithm to do that goes back to extensive work of F.\ Kirwan around 1985 (see \cite{Kirwan1},\cite{Kirwan4},\cite{Kirwan2},\cite{Kirwan3}) and is impracticable. Moreover, using wall-crossing formulas, we are now able to understand the change of intersection Hodge numbers under variations of stability conditions.\\
For the next corollary we mention that the moduli space of semistable quiver representations admits a proper map to the affine, connected moduli space of semisimple representations of the same dimension vector. If the quiver is acyclic, there is only one such semisimple representation. Thus,  the moduli space $\Msp^{ss}_d$ must be compact.
\begin{corollary}[Positivity]
 If $Q$ is acyclic  and the  stability condition generic, the (motivic) Donaldson--Thomas invariant $\DT_d$ is a palindromic polynomial in the Lefschetz motive with positive coefficients. 
\end{corollary}
Indeed, it is not hard to see that $\DT_d$ is always a rational function in the square root $\LL^{1/2}$ of the Lefschetz motive. Due to our main result, it must actually be a polynomial in the Lefschetz motive (up to normalization). By compactness (and normalization), $\IC^k(\Msp^{ss}_d,\QQ)$ carries a Hodge structure of weight $k$, and this can only happen for even $k$ as there are no Lefschetz motives in odd degree. The hard Lefschetz theorem implies that $\DT_d$ is a palindromic polynomial.\\
The next result is a direct consequence of our main theorem, Proposition \ref{localDT} and Corollary \ref{localDT2}.
\begin{corollary}[Locality]
 Fix a generic stability condition and a closed point $x\in \Msp^{ss}$, that is, a polystable complex representation  $V=\bigoplus_{k\in K} E_k^{m_k}$ of $Q$ with stable summands $E_k$. If the moduli space also contains stable representations, then the fiber at $x$ of the intersection complex of the moduli space is given by a certain Donaldson--Thomas invariant for the $\Ext$-quiver of the collection $E=(E_k)_{k\in K}$.  
\end{corollary}

Finally, we will give, in Theorem \ref{mi}, an explicit formula for the intersection Betti numbers of the classical spaces of matrix invariants (that is, the quotient of tuples of linear operators by simultaneous conjugation), using the explicit formula for motivic DT invariants for loop quivers in \cite{Reineke4}.\\[1ex]
The paper is organized as follows. Section 2 provides some background on quivers and their representations. The main purpose is to fix the notation. Although we will not use it, subsection 2.1 also contains a quick link to 3-Calabi--Yau categories - the natural environment of Donaldson--Thomas theory. The most important result of section 2 is Theorem \ref{virtsmall},  stating that the so-called Hilbert--Chow morphism from the moduli space $\Msp^{ss}_{f,d}$ of framed representations to the moduli space $\Msp^{ss}_d$ of unframed representations is what we will call virtually small.
\begin{theorem} For a generic stability condition and a dimension vector $d$, the Hilbert--Chow morphism $\pi:\Msp^{ss}_{f,d} \longrightarrow \Msp^{ss}_d$ is projective and virtually small, that is, there is a finite stratification $\Msp^{ss}_d=\sqcup_\xi S_\xi$ with empty or dense stratum $S_0=\Msp^{st}_d$ such that $\pi^{-1}(S_\xi) \longrightarrow S_\xi$ is \'etale locally trivial and 
 \[ \dim \pi^{-1}(x_\xi) - \dim \PP^{f\cdot d-1} \le \frac{1}{2} \codim S_\xi\] 
for every $x_\xi\in S_\xi$ with equality only for $S_\xi=S_0\not=\emptyset$ with fiber $\pi^{-1}(x_0)\cong \PP^{f\cdot d-1}$.
\end{theorem}
The proof of this important technical result is postponed to section 5 to keep the introduction short.\\

Section 3 is devoted to intersection complexes and the Schur functor formalism. As we need a nontrivial Lefschetz  ``motive'' $\LL$, restricting to perverse sheaves is not sufficient. Hence, we have to consider mixed Hodge modules, but there is no reason to be worried about that. We only need that the Grothendieck group is freely generated as a $\ZZ[\LL^{\pm 1}]$-module by some sort of intersection complexes. The (relative) hard Lefschetz theorem and some weight estimation for virtually small maps will also play a role. \\
Taking direct sums of representations induces a symmetric monoidal tensor product on the category of mixed Hodge modules by convolution. Using some general machinery (see \cite{Deligne1}), one can introduce Schur (endo)functors. Among them the symmetric and alternating powers are the most famous ones, and we finally obtain a $\lambda$-ring structure on the Grothendieck group of mixed Hodge structures. \\

The latter is used in Section 4 to define  Donaldson--Thomas functions. We will relate Donaldson--Thomas functions to framed quiver representations my means of the so-called DT/PT correspondence proven in section 6. Using this, the virtual smallness of the Hilbert--Chow morphism and the (relative) hard Lefschetz theorem, we finally deliver the proof of our main theorem by comparing degrees of polynomials in $\ZZ[\LL^{\pm 1/2}]$.\\ 

While proving our main result in section 4, we will observe that a certain integrality condition is crucial. It turns out that this condition is a relative version of the famous integrality conjecture in Donaldson--Thomas theory. Fortunately, we can give a proof in our situation of quiver representations by reducing the problem to a result of Efimov (see \cite{Efimov}, Theorem 1.1). In fact, the arguments use only the cut and paste relation allowing us to generalize the setting to motivic functions and to arbitrary ground fields of characteristic zero. Here is the main result of the second part of our paper, that is, of section 6.

\begin{theorem}[Integrality Conjecture, relative version]  
For a generic stability condition and a not necessarily closed point $x\in \Msp^{ss}$ there is a finite extension $\KK\supset \kk(x)$ of the residue field of $x$ giving rise to a map $\tilde{x}:\Spec\KK\to \Msp^{ss}$ such that the ``value'' $\DTS^{mot}(\tilde{x}):=\tilde{x}^\ast \DTS^{mot}$ of the motivic Donaldson--Thomas function at $\tilde{x}$  is in the image of the natural map 
\[ \Ka_0(\Var/\KK)[\LL^{-1/2}] \longrightarrow \Ka_0(\Var/\KK)[\LL^{-1/2}, (\LL^r-1)^{-1}: r\ge 1].\] 
\end{theorem}
Ideally, we would like to replace $\KK$ with $\kk(x)$ and $\tilde{x}$ with $x$, but we have good reasons to belief that such a result cannot hold for ``naive'' motives.\\
Similar to the Hodge realization, the Donaldson--Thomas invariant $\DT^{mot}_d$ is a rational function in $\LL^{1/2}$ with integer coefficients. Moreover, the coefficients are independent of the ground field and remain the same in any ``realization'' of motives. Using our main result on intersection complexes, we get the famous integrality conjecture.
\begin{corollary}[Integrality Conjecture, absolute version] 
For a generic stability condition the motivic Donaldson--Thomas invariant $\DT^{mot}_d$ is in the image of the natural map 
\[ \Ka_0(\Var/\kk)[\LL^{-1/2}] \longrightarrow \Ka_0(\Var/\kk)[\LL^{-1/2}, (\LL^r-1)^{-1}: r\ge 1].\]
\end{corollary}

This result has been obtained by Efimov for representations of symmetric quivers  and trivial stability condition (see  \cite{Efimov}, Theorem 1.1). A very complicated proof of the integrality conjecture even for quivers with potential was sketched by Kontsevich and Soibelman (see \cite{KS2}, Theorem 10).\\

\textbf{Acknowledgments.} The main result of the paper was originally observed and conjectured by J.\ Manschot while doing some computations. The first author is very grateful to him for sharing his observations and his conjecture which was the starting point of this paper.  The authors would also like to thank V.\ Ginzburg, E.\ Letellier, M.\ Kontsevich and L.\ Migliorini for interesting discussions on the results of this paper and J\"org Sch\"urmann for answering patiently all  questions about mixed Hodge modules.

\section{Moduli spaces of quiver representations}

\subsection{Quiver representations}

We fix a field $\KK$ which might either be our ground field $\kk$ or, as in section 6, a not necessarily algebraic extension of the latter. Let $Q=(Q_0,Q_1,s,t)$ be a quiver consisting of a finite set $Q_0$ of vertices, a finite set $Q_1$ of arrows as well as source and target maps $s,t:Q_1\to Q_0$. To any quiver we associate its path algebra $\KK Q$. The underlying $\KK$-vector space is spanned by paths of arbitrary length with a path of length zero attached to every vertex. Multiplication on $\KK Q$ is given by $\KK$-linear extension of concatenating paths. Equivalently, one could think of $\KK Q$ as a $\KK$-linear category with set of objects $Q_0$ and $\Hom_{\KK Q}(i,j)$ being the $\KK$-vector space generated by all paths from $i$ to $j$. Again, composition is induced by $\KK$-linear extension of concatenation. \\
There is a second (dg-)algebra associated to $Q$, namely its Ginzburg algebra $\Gamma_\KK Q$. The underlying algebra is the path algebra $\KK Q^{ex}$ associated to the extended quiver $Q^{ex}=(Q_0, Q_1\sqcup Q_1^{op} \sqcup Q_0, s^{ex},t^{ex})$ obtained from $Q$ by adding to every arrow $\alpha:i\to j$ of $Q$ another arrow $\alpha^\ast:j \to i$ with opposite orientation, and a loop $l_i:i\to i$ for every vertex $i\in Q_0$. We make $\Gamma_\KK Q$ into a dg-algebra by introducing a grading such that $\deg(\alpha)=0, \deg(\alpha^\ast)=-1$, and $\deg(l_i)=-2$. The differential is uniquely determined by putting 
\[ d\alpha=d\alpha^\ast=0 \quad\mbox{ and }\quad dl_i=\sum_{\alpha:i \to j} \alpha^\ast \alpha - \sum_{\alpha:j \to i} \alpha \alpha^\ast. \]      
Again, we can think of $\Gamma_\KK Q$ as a dg-category with set of objects being $Q_0$. Moreover, $H^0(\Gamma_\KK Q)\cong \KK Q$ can be interpreted as a dg-category with zero grading and trivial differential. \\
By looking at dg-functors $V:\KK Q \longrightarrow \dgV_\KK$ and $W: \Gamma_\KK Q  \longrightarrow \dgV_\KK$ into the category of dg-vector spaces with finite dimensional total cohomology, we get two dg-categories with model structures and associated triangulated homotopy ($A_\infty$-)categories $D^b(\KK Q\rep)$ and $D^b(\Gamma_\KK Q\rep)$. Each has a bounded t-structure with heart $\KK Q\rep$ being the abelian category of quiver representations, that is, of functors $V:\KK Q \longrightarrow \Vect_\KK$ into the category of finite dimensional $\KK$-vector spaces. In particular, 
\[ \Ka_0( D^b(\KK Q\rep)) \cong \Ka_0 (D^b(\Gamma_\KK Q  \rep)) \cong \Ka_0 (\KK Q\rep). \]
There is a group homomorphism $\dim: \Ka_0(\KK Q \rep)\longrightarrow \ZZ^{Q_0}$ associating to every representation resp.\ functor $V$ the tuple $(\dim_\KK V_i)_{i\in Q_0}$ of dimensions of the vector spaces $V_i:=V(i)$. There are two pairings on $\ZZ^{Q_0}$ defined by \begin{eqnarray*}
(d,e)& := & \sum_{i\in Q_0} d_ie_i \; - \sum_{Q_1\ni \alpha:i\to j} d_ie_j \\
\langle d,e \rangle &:=& (d,e)\;-\;(e,d) 
\end{eqnarray*}
such that the pull-back of these pairings via $\dim$ is just the Euler pairing induced by $D^b(\KK Q\rep)$ resp.\ $D^b(\Gamma_\KK Q\rep)$. The skew-symmetry of the latter reflects the fact that $D^b(\Gamma_\KK Q\rep)$ is a 3-Calabi--Yau category, that is, the triple shift functor $[3]$ is a Serre functor. This provides the link to Donaldson--Thomas theory.

\subsection{Moduli spaces}
The stack of $Q$-representations, that is, of objects in $\KK Q\rep$, can be described quite easily. Fix a dimension vector $d=(d_i)\in \NN^{Q_0}$ and note that $G_d:=\prod_{i\in Q_0} \Aut_\KK(\KK^{d_i})$ acts on $R_d:=\prod_{\alpha:i\to j} \Hom_\KK(\KK^{d_i},\KK^{d_j})$ in a canonical way by simultaneous conjugation. The stack of $Q$-representations of dimension $d$ is just the quotient stack $\Mst_d=R_d/G_d$. There are also derived (higher) stacks of objects in $D^b(\KK Q\rep)$ resp.\ $D^b(\Gamma_\KK Q\rep)$ containing $\Mst_d$ as a substack, but we are not going into this direction. \\
Instead, we want to study semistable representations of $Q$. As the radical of the Euler pairing contains the kernel of $\dim:\Ka_0(\KK Q\rep)\longrightarrow \ZZ^{Q_0}$, every tuple $\zeta=(\zeta_i)_{i\in Q_0}\in \{r\exp(i\pi\phi)\in \mathbb{C}\mid r>0, 0<\phi\le 1\}^{Q_0}\subset \mathbb{C}^{Q_0}$ provides a numerical Bridgeland stability condition on $D^b(\KK Q\rep)$ and on $D^b(\Gamma_\KK Q \rep)$ with central charge $Z(V)=\zeta\cdot \dim V:=\sum_{i\in Q_0}\zeta_i\dim_\KK V_i$ of slope $\mu(V):=- \Re e Z(V)/ \Im m Z(V)$ and standard t-structure. Hence we get an open substack $\Mst^{ss}_d=R^{ss}_d/G_d$ of semistable $Q$-representations. 
For every $\mu\in (-\infty,+\infty]$ let $\Lambda_\mu\subset \NN^{Q_0}$ be the monoid of dimension vectors $d$ (including $d=0$) such that $\zeta\cdot d=\sum_{i\in Q_0}\zeta_id_i\in \CC$ has slope $\mu$. We call $\zeta$ $\mu$-generic if $\langle d,e\rangle =0$ for all $d,e\in \Lambda_\mu$, and generic if that holds for all $\mu$. The non-generic ``stability conditions'' $\zeta$ lie on a countable but locally finite union of walls in $\{r\exp(i\pi\phi)\in \mathbb{C}\mid r>0, 0<\phi\le 1\}^{Q_0}$ of real codimension one. Obviously every stability for a symmetric quiver is generic. Another important class is given by complete bipartite quivers and the maximally symmetric stabilities used in \cite{RW} to construct a correspondence between the cohomology of quiver moduli and the GW invariants of \cite{GPS}.\\    
As we wish to form moduli schemes, we should restrict ourselves to King stability conditions $\zeta=(-\theta_i + \sqrt{-1})_{i\in Q_0}$ for some $\theta=(\theta_i)\in \ZZ^{Q_0}$, giving rise to a linearization of the $G_d$ action on $R_d$ with semistable points $R^{ss}_d$. Let us denote the GIT quotient by $\Msp^{ss}_d= R^{ss}_d/\!\!/ G_d$. The points $x$ in $\Msp^{ss}_d$ correspond to polystable representations $V=\bigoplus_{k\in K} E_k$ defined over some finite extension of the residue field of $x$. The obvious morphism $p:\Mst^{ss}_d \longrightarrow \Msp^{ss}_d$ maps a semistable representation to the direct sum of its stable factors. We also have the open substack $\Mst^{st}_d\subset \Mst^{ss}_d$ of stable representations mapping to  the open subvariety $\Msp^{st}_d\subset \Msp^{ss}_d$ of stable representations. 
Note that $\Mst_d, \Mst^{ss}_d, \Mst^{st}_d,$ and $\Msp^{st}_d$ are smooth while $\Msp^{ss}_d$ is not. Moreover, $\Msp^{st}_d$ is either dense in $\Msp^{ss}_d$ or empty. We call $\theta$ ($\mu$-)generic if $\zeta=(-\theta_i + \sqrt{-1})_{i\in Q_0}$ is ($\mu$-)generic in the previous sense. \\
The construction of coarse moduli spaces can also be done for so-called geometric Bridgeland stability conditions, i.e.\ for $\zeta$  not lying on a (different) countable union of real codimension one walls. Indeed, given $\zeta$ and a dimension vector $d$, we can always perturb $\zeta$ slightly to $\zeta'$ with rational real and imaginary part  without changing $R^{ss}_d$. This is true because $R^{ss}_d$ will only change if $\zeta$ crosses a finite subset (depending on $d$) of these walls. Given $\zeta'=a+b\sqrt{-1}$ with $a,b\in \mathbb{Q}^{Q_0}$, we may define $\theta:=N\big((a\cdot d)\, b - (b\cdot d)\,a \big)$ with $N\gg 0$ such that $\theta\in \ZZ^{Q_0}$. Then, $\theta\cdot d=0$. Moreover, $\theta\cdot d'\le 0$ if and only if $\arg Z'(d')\le \arg Z'(d)$ if and only if $\arg Z(d')\le Z(d)$ for all nonzero dimension vectors $d'$ smaller than $d$. Hence, $R^{ss}_d$ is the open subset of semistable points in the GIT sense, and a categorical quotient $\Msp^{ss}_d:=R^{ss}_d/\!\!/G_d$ exists. As the latter satisfies a universal property, it does not depend on the choice of $\zeta'$ and $r\ge 1$. From now on, we will always assume that $\zeta$ is geometric so that moduli spaces exist.\\
We use the notation $\Msp^{ssimp}_d$ for the King stability condition $\theta=0$. Points in $\Msp^{ssimp}_d$ correspond to semisimple representations of dimension $d$. For every stability condition there is a projective morphisms $\Msp^{ss}_d\to \Msp^{ssimp}_d$ mapping  any (polystable) representation to the sum of its Jordan--H\"older factors taken in $\KK Q\rep$.\\

Given two dimension vectors $d,d'$, we denote with $R_{d,d'}$ the (linear) subvariety of $R_{d+d'}$ corresponding to representations which preserve the subspace $\KK^d\subset \KK^{d}\oplus \KK^{d'}=\KK^{d+d'}$. Similarly, $G_{d,d'}\subset G_{d+d'}$ is the subgroup preserving this subspace. Then, $\mathfrak{E}xact_{d,d'}=R_{d,d'}/G_{d,d'}$ is the stack of short exact sequences of representations with prescribed dimensions of the outer terms. There are morphisms $\pi_1\times \pi_2\times\pi_3: \mathfrak{E}xact_{d,d'}\longrightarrow \Mst_d\times\Mst_{d+d'}\times \Mst_{d'}$ mapping a sequence to the corresponding entry. Note that $\pi_2$ is the universal quiver Grassmannian for $Q$, hence representable and proper. In particular, $\mathfrak{E}xact_{d,d'}\cong Y_{d,d'}/G_{d+d'}$ for $Y_{d,d'}=R_{d,d'}\times_{G_{d,d'}} G_{d+d'}$. \\
Let us continue this section with a simple but important observation. Given a slope $\mu\in (-\infty,+\infty]$, the moduli stack $\Mst^{ss}_\mu:= \sqcup_{d\in \Lambda_\mu} \Mst^{ss}_d$, resp.\ the moduli space $\Msp^{ss}_\mu:= \sqcup_{d\in \Lambda_\mu} \Msp^{ss}_d$, is a commutative monoid in the category of stacks, resp.\ schemes, with respect to direct sums of representations. The unit is given by the zero-dimensional representation which is considered to be semistable with any slope. Obviously, the morphisms $p: \Mst^{ss}_\mu \longrightarrow \Msp^{ss}_\mu$ and $\dim:\Msp^{ss}_\mu \longrightarrow \Lambda_\mu$ mapping every polystable representation to its dimension vector are monoid homomorphism. 

\begin{lemma}
The morphism $\oplus:\Msp^{ss}_\mu\times \Msp^{ss}_\mu \longrightarrow \Msp^{ss}_\mu$ is finite. 
\end{lemma}
\begin{proof}
As the isomorphism types and multiplicities of the stable summands of a polystable object are unique, the morphism is certainly quasi-finite. It remains to show that $\oplus$ is proper. There is a commutative diagram
\[ \xymatrix @C=2cm { \Msp^{ss}_\mu\times \Msp^{ss}_\mu \ar[r]^\oplus \ar[d] & \Msp^{ss}_\mu \ar[d] \\ \Msp^{ssimp}_\mu \times \Msp^{ssimp}_\mu \ar[r]^\oplus & \Msp^{ssimp}_\mu }\]
with proper vertical maps. Hence, it suffices to show that $\oplus:\Msp^{ssimp}_\mu \times \Msp^{ssimp}_\mu \longrightarrow \Msp^{ssimp}_\mu$ is proper.
Consider the following commutative diagram
\[ \xymatrix { & \mathfrak{E} xact_{d,d'}\cong Y_{d,d'} /G_{d+d'} \ar@/_/[dl]_{\pi_1\times \pi_3} \ar[dr]^{\pi_2}  \ar[dd]^{\rho_{d,d'}} & \\ R_d/G_d \times R_{d'}/G_{d'}, \ar@/_/[ur]_{\sigma_0} \ar[dd]_{\rho_d\times \rho_{d'}} & & R_{d+d'}/G_{d+d'}  \ar[dd]^{\rho_{d+d'}} \\
& \Spec \kk[Y_{d,d'}]^{G_{d+d'}}  \ar@/_/[dl]_{\tilde{\pi}_1\times \tilde{\pi}_3} \ar[dr]^{\tilde{\pi}_2} &  \\ 
\Spec \kk[R_d]^{G_d} \times \Spec \kk[R_{d'}]^{G_{d'}} \ar@/_/[ur]_{\tilde{\sigma}_0} \ar[rr]_\oplus & & \Spec \kk[R_{d+d'}]^{G_{d+d'}} }  \]
with $Y_{d,d'}\cong R_{d,d'}\times_{G_{d,d'}}G_{d+d'}\cong \mathfrak{E} xact_{d,d'}\times_{\Mst_{d+d'}}R_{d+d'}$. Here, $\sigma_0$ maps a pair $(V,V')$ of representations to its direct sum $V\oplus V'$ providing a right inverse of $\pi_1\times \pi_3$. Thus, $\tilde{\sigma}_0$ is also a section providing a closed embedding. It remains to show that $\tilde{\pi}_2$ is proper. Note that $\hat{\pi}_2:Y_{d,d'}\longrightarrow R_{d+d'}$, being the pull-back of $\pi_2$, must be proper with Stein factorization $Y_{d,d'} \to \Spec\kk[Y_{d,d}] \to R_{d+d'}$ as $R_{d+d'}$ is affine. Thus, $\kk[R_{d+d'}]\longrightarrow \kk[Y_{d,d'}]$ is finite, hence integral. Applying the Reynolds operator of $\kk[Y_{d,d'}]$ to an integral equation for $a\in \kk[Y_{d,d'}]^{G_{d+d'}}$, we obtain that $\kk[R_{d+d'}]^{G_{d+d'}}\longrightarrow \kk[Y_{d,d'}]^{G_{d+d'}}$ is integral, too. Thus $\tilde{\pi}_2$ is finite, hence proper.
\end{proof}

For later applications we also need framed $Q$-representations (see \cite{Reineke1}). We fix a framing vector $f\in \NN^{Q_0}$ and consider representations of a new quiver $Q_f=(Q_0\sqcup\{\infty\}, Q_1\sqcup \{\beta_{l_i}:\infty \to i \mid i\in Q_0, 1\le l_i\le f_i \})$ with dimension vector $d'$ obtained by extending $d$ via $d_\infty=1$. We also extend $\zeta$ appropriately (see \cite{Reineke1}) and get a King stability condition $\zeta'$ for $Q_f$. Let $\Msp^{ss}_{f,d}$ be the moduli space of $\zeta'$-semistable $Q_f$-representations of dimension vector $d'$. It turns out that $\Msp^{ss}_{f,d}=\Msp^{st}_{f,d}$, and thus $\Msp^{ss}_{f,d}$ is smooth and $p_{f,d}:\Mst^{ss}_{f,d}\to \Msp^{ss}_{f,d}$  a principal bundle with structure group $P(G_d\times \GG_m)\cong G_d$. There is an obvious morphism $\pi:\Msp^{ss}_{f,d} \longrightarrow \Msp^{ss}_d$ obtained by restricting a $\zeta'$-(semi)stable representation of $Q_f$ to the subquiver $Q$ which turns out to be $\zeta$-semistable. The following theorem will we crucial for proving our main result. To keep the introduction short, we will postpone its proof to 
section 5.    

\begin{theorem} \label{virtsmall} Let $\mu$ be the slope of a dimension vector $d$ with respect to a stability condition $\zeta$. If $\zeta$ is $\mu$-generic, the morphism $\pi:\Msp^{ss}_{f,d} \longrightarrow \Msp^{ss}_d$ is projective and virtually small, that is, there is a finite stratification $\Msp^{ss}_d=\sqcup_\xi S_\xi$ with empty or dense stratum $S_0=\Msp^{st}_d$ such that $\pi^{-1}(S_\xi) \longrightarrow S_\xi$ is \'etale locally trivial and 
 \[ \dim \pi^{-1}(x_\xi) - \dim \PP^{f\cdot d-1} \le \frac{1}{2} \codim S_\xi\] 
for every $x_\xi\in S_\xi$ with equality only for $S_\xi=S_0\not=\emptyset$ with fiber $\pi^{-1}(x_0)\cong \PP^{f\cdot d-1}$.
\end{theorem}

Let us also introduce the notation $\Msp^{ss}_{f,\mu}:=\sqcup_{d\in \Lambda_\mu} \Msp^{ss}_{f,d}$ and $\Msp^{st}:= \sqcup_{0\neq d\in \NN^{Q_0}} \Msp^{st}_d$.

\section{Intersection complex}

\subsection{From perverse sheaves to mixed Hodge modules} 

The ground field in the next two sections will be $\kk=\CC$. In this section we recall some standard facts about perverse sheaves, intersection complexes and Schur functors. The interested reader will find more details in \cite{CaMi} and \cite{Saito1}. Let $X$ be a variety with quasiprojective connected components. We denote with $\Perv(X)$ resp.\ $\MHM(X)$ the abelian categories of perverse sheaves resp.\ mixed Hodge modules on $X$. There is a natural functor $\rat:\MHM(X) \longrightarrow \Perv(X)$ associating to every mixed Hodge module its underlying perverse sheaf. For a morphism $f:X\longrightarrow Y$ of finite type we get two pairs $(f^\ast,f_\ast), (f_!,f^!)$ of adjoint triangulated functors $f_\ast,f_!:D^b(\Perv(X)) \longrightarrow D^b(\Perv(Y))$ and $f^\ast,f^!:D^b(\Perv(Y)) \longrightarrow D^b(\Perv(X))$, and similarly for mixed Hodge modules,  satisfying Grothendieck's axioms of the four functor formalism. Moreover, the functor $\rat$ is compatible with these functors in the obvious way, and there 
are duality functors relating $f_\ast$ with $f_!$ and $f^\ast$ 
with $f^!$. We also mention that for each connected component $X_\alpha$ of $X$, the categories $\Perv(X_\alpha)$ and $\MHM(X_\alpha)$ are of finite length. Furthermore, there is an element $\TT$ of $\MHM(\CC)$, called the Tate object. Since $\MHM(\CC)$ acts on $\MHM(X)$, we get an exact autoequivalence on $D^b(\MHM(X))$, abusing notation also denoted with $\TT$, given by multiplication with $\TT$. It commutes with all four functors and 
satisfies $\rat\circ \TT = \rat$. In our case, $X$ will carry the structure of a commutative monoid with unit $0\in X$, and $\MHM(\CC)$ can be interpreted as the subcategory of mixed Hodge modules supported at $0$. The action of $\MHM(\CC)$ on $\MHM(X)$  is induced by the  convolution product on $\MHM(X)$ which we introduce later.      
The actions of $\TT$ and $\LL:=\TT[-2]$ on $\Ka_0(\MHM(X))$ coincide, making it into a $\ZZ[\LL^{\pm 1}]$-module. We denote with $\Ka_0(\MHM(X))[\LL^{-1/2}]$ the $\ZZ[\LL^{\pm 1/2}]$-module obtained by adjoining a square root of $\LL$. One can also categorify this, giving rise to a square root $\TT^{1/2}$ of $\TT$ in an enlarged abelian category of mixed Hodge motives. Then, $\LL^{-1/2}=\TT^{-1/2}[1]$, and one should interpret the multiplication with $\LL^{-1/2}$ as a refinement of the shift functor $[1]$ on $D^b(\Perv(X))$.

\subsection{Intersection complex}
Given a closed equidimensional subvariety $Z\subset X$ and a local system on a dense open subset $Z^o$ of the regular part $Z_{reg}$ of $Z$, there is canonical perverse sheaf $\ICS_Z(L)$ on $X$, called the $L$-twisted intersection complex of $Z$, such that $\ICS_Z(L)|_{Z^o}=L[\dim Z]$. If $Z$ and $L$ are irreducible, $\ICS_Z(L)$ is an irreducible object of $\Perv(X)$, and all irreducible objects are obtained in this way. For $\MHM(X)$, there is a similar construction, with $L$ replaced with a (graded) polarizable, admissible
variation of (mixed) Hodge structures $L$ with quasi-unipotent monodromy at ``infinity''. We will, however, use the slightly non-standard normalization $\ICS_Z(L)|_{Z^o}=\LL^{-\dim Z/2}L$ with the convention that $\rat(L)$ is the unshifted local system given by $L$. As $\rat(\LL^{-\dim Z/2})=\QQ[\dim Z]$, the usual shift  in the de Rham functor is not lost but ``absorbed'' by the normalization factor. Note that an irreducible variation of mixed Hodge structures is pure, and application of $\TT^{-1/2}$ reduces the weight by one. If $Z$ has several connected components of different dimension, the construction of $\ICS_Z(L)$ generalizes accordingly. Applying this to the trivial variation $\QQ$ of pure Hodge structures of type $(0,0)$ on 
$Z_{reg}$, we obtain a distinguished intersection complex $\ICS_Z(\QQ)$.       

\subsection{Schur functors}
Let us now specialize to $X=\Msp^{ss}_\mu$, although everything in this section remains true for arbitrary commutative monoids $(X,\oplus,0)$ in the category of varieties with quasiprojective connected components such that $\oplus:X\times X \longrightarrow X$ is finite. Due to the last property, the higher derived direct images $R^i\oplus_\ast$ vanish, and we obtain a symmetric monoidal tensor product
\[ \otimes: \MHM(\Msp^{ss}_\mu)\times \MHM(\Msp^{ss}_\mu) \longrightarrow \MHM(\Msp^{ss}_\mu), \quad \EE\otimes \FF:=\oplus_\ast(\EE\boxtimes \FF), \]
and similarly for $\Perv(\Msp^{ss}_\mu)$. The unit $\unit$ is given by $\ICS_{\Msp^{ss}_0}(\QQ)$, which is a skyscraper sheaf of rank one supported at the zero-dimensional representation $0$. More details can be found in \cite{Schuermann}. We drop the $\otimes$-sign when dealing with the associated Grothendieck groups $\Ka_0(\Perv(\Msp^{ss}_\mu))$ and $\Ka_0(\MHM(\Msp^{ss}_\mu))$, respectively. \\
Given $\EE\in \MHM(\Msp^{ss}_\mu)$ and $n\in \NN$, the mixed Hodge module $\EE^{\otimes n}$ carries a natural action of the symmetric group $S_n$. By general arguments (see \cite{Deligne1}), we obtain a decomposition 
\[ \EE^{\otimes n} =  \bigoplus\limits_{\lambda \dashv n} W_\lambda \otimes S^\lambda(\EE) \]
for certain mixed Hodge modules $S^\lambda(\EE)$, where $W_\lambda$ denotes the irreducible representation of $S_n$ associated to the partition $\lambda$ of $n$. The tensor product used on the right hand side can be defined for every additive category, and should not be confused with the tensor product explained above. However, after identifying vector spaces $W$ with trivial variations of pure Hodge structures of type $(0,0)$ over $\Msp^{ss}_0$, both tensor products agree. The decomposition is functorial,  giving rise to Schur functors $S^\lambda:\MHM(\Msp^{ss}_\mu) \longrightarrow \MHM(\Msp^{ss}_\mu)$ for every partition $\lambda$. The same construction also applies to $\Perv(\Msp^{ss}_\mu)$, and $\rat:\MHM(\Msp^{ss}_\mu)\longrightarrow \Perv(\Msp^{ss}_\mu)$ ``commutes'' with Schur functors of the same type.    
\begin{example} \rm \quad 
\begin{enumerate}
 \item For $\lambda=(n)$, the representation $W_\lambda$ is the trivial representation of $S_n$ and $S^\lambda(\EE)=:\Sym^n(\EE)$. If $\EE|_{\Msp^{ss}_0}=0$, we get  $\Sym^n(\EE)|_{\Msp^{ss}_d}=0$ for every $d\in \Lambda_\mu$ provided $n\gg 0$. In particular, $\Sym(\EE)=\oplus_n \Sym^n(\EE)$ is well-defined. 
  \item For $\lambda=(1,\ldots,1)$, the representation $W_\lambda$ is the sign representation of $S_n$ and $S^\lambda(\EE)=:\Alt^n(\EE)$. As before $\Alt(\EE)=\oplus_n \Alt^n(\EE)$ is well-defined provided $\EE|_{\Msp^{ss}_0}=0$. 
\end{enumerate}
\end{example}
The following proposition is a standard result. 
\begin{proposition}
 Let $\EE,\FF$ be in $\MHM(\Msp^{ss}_\mu)$ or in $\Perv(\Msp^{ss}_\mu)$ such that $\EE|_{\Msp^{ss}_0}=\FF|_{\Msp^{ss}_0}=0$. Denote with $\Part$ be the set of all partitions of arbitrary size. Then
 \begin{eqnarray}
  \Sym(\EE \oplus \FF) & = & \Sym(E) \otimes \Sym(F) , \;\mbox{ in particular} \nonumber\\
  \Sym^n(\EE \oplus \FF) & = & \bigoplus\limits_{i+j=n} \Sym^i(\EE)\otimes \Sym^j(\FF),\;\mbox{and} \label{eqn1} \\
  \Sym (\EE \otimes \FF) & = & \bigoplus\limits_{\lambda \in \Part } S^\lambda(\EE)\otimes S^\lambda(\FF), \;\mbox{ in particular} \nonumber  \\
  \Sym^n(\EE\otimes \FF) & = & \bigoplus\limits_{\lambda \dashv n} S^\lambda(\EE)\otimes S^\lambda(\FF). \label{eqn2}
 \end{eqnarray}
\end{proposition}
Equations (\ref{eqn1}) and (\ref{eqn2}) are of course also true without the additional assumptions on $\EE$ and $\FF$. The next result is also well-known.
\begin{proposition} \label{lambda-ring}
 The Schur functors $S^\lambda$ induce well defined operations on the Grothendieck groups $\Ka_0(\Perv(\Msp^{ss}_\mu))$ and $\Ka_0(\MHM(\Msp^{ss}_\mu))$, respectively, satisfying the analogs of equation (\ref{eqn1}) and (\ref{eqn2}). In particular, both Grothendieck groups carry the structure of a (special) $\lambda$-ring. 
\end{proposition}
It is worth to mention the following technical detail. Although $\Sym(\EE)=\oplus_n \Sym^n(\EE)$ by definition, this equation cannot hold on the level of Grothendieck groups as we do not have infinite sums. To define these, we need to complete the Grothendieck groups as follows. Let $F^p\subset \Ka_0(\Perv(\Msp^{ss}_\mu))$ be the subgroup generated by all perverse sheaves $\EE$ such that $\EE|_{\Msp^{ss}_d}=0$ if $d$ cannot be written as a sum of $p$ nonzero dimension vectors, i.e.\ $|d|:=\sum_{i\in Q_0}d_i <p$. It is easy to these that $F^pF^q\subset F^{p+q}$  and $S^\lambda(F^p)\subset F^{np}$ for all $\lambda\dashv n$ and all $n,p,q\in \NN$. Hence, the $F^p$ provide a $\lambda$-ring filtration, and the corresponding completion $\underline{\Ka}_0(\Perv(\Msp^{ss}_\mu))=\prod_{d\in \Lambda_\mu}\Ka_0(\Perv(\Msp^{ss}_d))$ has a well defined ring structure and action of $S^\lambda$. Moreover, $\sum_n \Sym^n(\EE)$ is well-defined and agrees with the class of $\Sym(\EE)$ for $\EE\in F^1$. The completion of $\Ka_0(\MHM(\Msp^{ss}_\mu))$ is done in the same way. \\
As $\TT=\LL$ in $\Ka_0(\MHM(\CC))$ and $\Sym^n(\TT^{\pm 1})=\TT^{\pm n}$, the $\lambda$-ring structure of Proposition \ref{lambda-ring} can be extended to $\Ka_0(\MHM(\Msp^{ss}_\mu))[\LL^{-1/2}]$, and even to 
\[ \Ka_0(\MHM(\Msp^{ss}_\mu))[\LL^{-1/2}, (\LL^r-1)^{-1}: r\ge 1]=\]
\[= \Ka_0(\MHM(\Msp^{ss}_\mu))\otimes_{\ZZ[\LL^{\pm 1}]} \ZZ[\LL^{-1/2}, (\LL^r-1)^{-1}: r\ge 1] \] such that 
\[ S^\lambda(\LL^{\pm 1/2})= \begin{cases} \LL^{\pm n/2} & \mbox{ for }\lambda=(1,\dots,1), \\ 0 & \mbox{ otherwise. } \end{cases} \]
Again, we consider the filtration $F^p[\LL^{-1/2}]$, resp.\ $F^p[L^{-1/2},(\LL^r-1)^{-1}: r\ge 1]$, defined accordingly and perform a completion as before. By abusing notation let us denote the resulting $\lambda$-ring with 
\[ \underline{\Ka}_0(\MHM(\Msp^{ss}_\mu))[\LL^{-1/2}, (\LL^r-1)^{-1}: r\ge 1] := \]
\[= \prod_{d\in \Lambda_\mu} \Bigl( \Ka_0(\MHM(\Msp^{ss}_d))\otimes_{\ZZ[\LL^\pm]}[\LL^{-1/2},(\LL^r-1)^{-1}:r\ge 1] \Bigr) \]
which should not be confused with 
\[ \Bigl( \prod_{d\in \Lambda_\mu} \Ka_0(\MHM(\Msp^{ss}_\mu))\Bigr) \otimes_{\ZZ[\LL^{\pm 1}]} \ZZ[\LL^{-1/2}, (\LL^r-1)^{-1}: r\ge 1].\] 
\begin{remark} \rm One reason for adjoining $\LL^{\pm 1/2}$ and our convention for intersection complexes is to symmetrize weight polynomials under Poincar\'{e} duality. Our choice of extending $S^\lambda$ is done in such a way that $\TT^{1/2}$ is again a line element. The various completions are needed in the next section when we pass to stacks and define Donaldson--Thomas invariants.
\end{remark}
The following result illustrates the nice behavior of intersection complexes with respect to Schur functors.
\begin{proposition}
Given a dimension vector $d$ with $\Msp^{st}_d\not=\emptyset$ and a natural number $n$, let us denote with $\Delta$ and $\tilde{\Delta}$ the big diagonal in $\Sym^n \Msp^{st}_d \subset \Msp^{ss}_{nd}$ and $(\Msp^{st}_d)^n$ respectively. For an irreducible representation $W_\lambda$ of $S_n$ denote with $\underline{W}_\lambda$ the variation of Hodge structure of type 
$(0,0)$ on $\Sym^n\Msp^{st}_d\setminus \Delta$ given by $\bigl((\Msp^{st}_d)^n\setminus \tilde{\Delta}\bigr) \times_{S_n} W_\lambda$. Then
\begin{equation} \label{eqn3} S^\lambda \bigl( \ICS_{\Msp_d^{ss}}(\QQ) \bigr) = \ICS_{Z_n}(\underline{W}_{\lambda^\ast})
\end{equation}
with $\lambda^\ast$ being the conjugate partition of $\lambda$ if $\dim \Msp^{st}_d=1-(d,d)$ is odd and $\lambda^\ast=\lambda$ if $\dim\Msp^{st}_d$ is even. Moreover, $Z_n$ is the irreducible closed image of $\oplus:(\Msp^{ss}_d)^n\to \Msp^{ss}_{nd}$.
\end{proposition}
\begin{proof}
Since $\oplus:(\Msp^{ss}_d)^n\longrightarrow \Msp^{ss}_{nd}$ is a small map,  $\ICS_{\Msp^{ss}_d}(\QQ)^{\otimes n}=\oplus_\ast \bigl(\ICS_{\Msp^{ss}_d}(\QQ)^{\boxtimes n}\bigr)=\ICS_{Z_n}(L)$ for a suitable variation of Hodge structures $L$ on the open smooth image $Z^o_n$ of  $(\Msp^{st}_d)^n\setminus \tilde{\Delta} \longrightarrow Z_n$. The latter map   induces an isomorphism between the geometric points of $\Sym^n \Msp^{st}_d\setminus \Delta$ and of $Z^o_n$. By Zariski's main theorem, $Z^o_n\cong \Sym^n \Msp^{st}_d\setminus \Delta$. As the restriction of $\oplus$ to $(\Msp^{st}_d)^n\setminus \tilde{\Delta}$ is a left principal $S_n$-bundle over $\Sym^n\Msp^{st}_d\setminus \Delta\cong Z^o_n$, we can trivialize it \'{e}tale locally as $U\times S_n$ with $U\to Z^o_n$ being the \'{e}tale cover $(\Msp^{st}_d)^n\setminus \tilde{\Delta} \longrightarrow Z^o_n$, showing that the fiber of $L$ is just $\LL^{-n\dim\Msp^{st}_d/2}\otimes H^0(S_n,\QQ)$. The natural $S_n$-action on $\ICS_{\Msp^{ss}_d}(\QQ)^{\otimes n}$ is induced by the left multiplication with $S_n$ on the second factor of $U\times S_n$, while the right multiplication on $S_n$ and on $U$ corresponds to the Galois action of this \'{e}tale cover giving rise to a nontrivial monodromy of $L$. The $S_n$-bimodule $H^0(S_n,\QQ)$ decomposes as $\oplus_{\lambda \dashv n} W_\lambda\otimes W_\lambda$ with the left  and  the right factor corresponding to the left and the right $S_n$-action respectively. Moreover, by our convention, $\LL^{-n\dim \Msp^{st}_d/2}$ carries the $\dim \Msp^{st}_d$-th power of the sign representation. Thus, $L=\oplus_{\lambda \dashv n} W_{\lambda^\ast} \otimes \underline{W}_\lambda=\oplus_{\lambda \dashv n} W_{\lambda} \otimes \underline{W}_{\lambda^\ast}$ completing the proof.      
\end{proof}
\begin{remark} \rm
 The occurrence of conjugate partitions looks rather unnatural but is related to the fact that the naive permutation action of $S_n$ on left D-modules needs to be twisted by the sign representation depending on the dimension. See \cite{Schuermann}, Remark 1.6(i) for more details.
\end{remark}

We can also replace $\Msp^{ss}_\mu$ with $\NN^{Q_0}\times \Spec\CC$ considered as a zero-dimensional monoid in the category of complex varieties with quasiprojective connected components. All of our constructions  go through, and it is not difficult to see that 
\[ \underline{\Ka}_0(\MHM(\NN^{Q_0}\times \Spec\CC))[\LL^{-1/2}, (\LL^r-1)^{-1}: r\ge 1] =\]
\[ \Ka_0(\MHM(\CC))[\LL^{-1/2}, (\LL^r-1)^{-1}: r\ge 1][[ t_i:i\in Q_0]]\]
is the ring of power series in $|Q_0|$ variables. Since $\dim:\Msp^{ss}_\mu \longrightarrow \NN^{Q_0}\times \Spec\CC$ is a homomorphism of monoids with $\oplus$ and $+$ being finite, $\dim_\ast$ and  $\dim_!$ define triangulated tensor functors $D^b(\MHM(\Msp^{ss}_\mu)) \longrightarrow D^b(\MHM(\NN^{Q_0}\times \Spec\CC))$ commuting with Schur functors of the same type. In particular, we get  $\lambda$-ring homomorphisms $\dim_\ast$ and $\dim_!$ from
\[ \underline{\Ka}_0(\MHM(\Msp^{ss}_\mu))[\LL^{-1/2}, (\LL^r-1)^{-1}: r\ge 1] \]
to
\[  \Ka_0(\MHM(\CC))[\LL^{-1/2}, (\LL^r-1)^{-1}: r\ge 1][[t_i:i\in Q_0]] \]
commuting with the Schur operators, and similarly for perverse sheaves.

\section{DT invariants and intersection complexes}

\subsection{Donaldson--Thomas invariants}

We will now introduce a generalization of Donaldson--Thomas invariants using the notation of the previous sections. Let us fix a slope $\mu\in (-\infty,+\infty]$ and consider the morphism $p:\Mst^{ss}_\mu \longrightarrow \Msp^{ss}_\mu$. Our first object is\footnote{Note that $p_!$ is the derived direct image with compact support, while $p_\ast$ is the usual derived direct image.} $p_! \ICS_{\Mst^{ss}_\mu}(\QQ)$ in $\underline{\Ka}_0(\MHM(\Msp^{ss}_\mu))[\LL^{-1/2}, (\LL^r-1)^{-1}: r\ge 1]$. To define it properly, we should develop a theory of mixed Hodge modules on Artin stacks along with a four functor formalism. However, in our situation of smooth quotient stacks we will use a more direct approach avoiding complicated machinery. First of all, $\Mst^{ss}_d$ is smooth,  motivating $\ICS_{\Mst^{ss}_d}(\QQ)=\LL^{-\dim \Mst^{ss}_d/2}\QQ=\LL^{(d,d)/2}\QQ$. Recall that $q:R^{ss}_d \longrightarrow \Mst^{ss}_d$ is a $G_d$-principal bundle for every dimension vector $d$. By means of the projection formula we 
would expect a formula like 
\[ H_c^\ast(G_d,\QQ) \,\ICS_{\Mst^{ss}_d}(\QQ) = q_! q^\ast \ICS_{\Mst^{ss}_d}(\QQ) = \LL^{\dim G_d/2} q_! \ICS_{R^{ss}_d}(\QQ)= \LL^{(d,d)/2}q_! \QQ \]
in $\underline{\Ka}_0(\MHM(\Mst^{ss}_\mu))[\LL^{-1/2}, (\LL^r-1)^{-1}: r\ge 1]$. Hence, we will define $p_! \ICS_{\Mst^{ss}_d}(\QQ)$ as the product in $\underline{\Ka}_0(\MHM(\Msp^{ss}_\mu))[\LL^{-1/2}, (\LL^r-1)^{-1}: r\ge 1]$ of $\LL^{(d,d)/2}p_! q_! \QQ$ with the inverse of the class $\prod_{i\in Q_0}\LL^{d_i \choose 2}\prod_{r=1}^{d_i}(\LL^r-1)\in \ZZ[\LL]\subset\Ka_0(\MHM(\CC))$ of $H_c^\ast(G_d,\QQ)$. ``Summing'' over $d\in \Lambda_\mu$ gives $p_! \ICS_{\Mst^{ss}_\mu}(\QQ)$. The following lemma is a standard fact in the theory of (filtered) $\lambda$-rings.
\begin{lemma}
 There is an element $\DTS_\mu\in \underline{\Ka}_0(\MHM(\Msp^{ss}_\mu))[\LL^{-1/2}, (\LL^r-1)^{-1}: r\ge 1]$ with $\DTS_\mu|_{\Msp_0^{ss}}=0$ such that
 \[ p_! \ICS_{\Mst^{ss}_\mu}(\QQ) = \Sym \Bigl( \frac{1}{\LL^{1/2}-\LL^{-1/2}} \DTS_\mu \Bigr).\]
\end{lemma}
\begin{definition}
 We call $\DTS\in \underline{\Ka}_0(\MHM(\Msp^{ss}))[\LL^{-1/2}, (\LL^r-1)^{-1}: r\ge 1]$ with $\DTS|_{\Msp^{ss}_\mu}=\DTS_\mu$ for all $\mu\in (-\infty,+\infty]$ the Donaldson--Thomas ``function'' and $\DT_d:=\dim_! \DTS_d = H^\ast_c(\Msp^{ss}_d,\DTS_d)\in \Ka_0(\MHM(\CC))[\LL^{-1/2}, (\LL^r-1)^{-1}: r\ge 1]$ the Donaldson--Thomas invariant of dimension vector $d$ with respect to the given stability condition $\zeta$.   
\end{definition}
As $\dim_!$ is a $\lambda$-ring homomorphism, our definition of Donaldson--Thomas invariants agrees with the usual one \cite{KS2}. 
Recall that our   stability condition $\zeta$ was called $\mu$-generic if $\langle d,e\rangle=0$ for all $d,e\in \Lambda_\mu$, and generic if that holds for all $\mu\in (-\infty,+\infty]$. The following result is Corollary \ref{alternative_form}. 
\begin{proposition}
For a $\mu$-generic stability condition and a framing vector $f\in \NN^{Q_0}$ such that $2|f_i$ for all $i\in Q_0$, we obtain the following formula with $\Lambda'_\mu:=\Lambda_\mu\setminus \{0\}$
\begin{equation} \label{eqn4} \pi_\ast \ICS_{\Msp^{ss}_{f,\mu}}(\QQ)=\pi_! \ICS_{\Msp^{ss}_{f,\mu}}(\QQ) = \Sym \Bigl( \sum_{d\in \Lambda'_\mu} [\PP^{f\cdot d-1}]_{vir} \DTS_d \Bigr) \end{equation}
in $\underline{\Ka}_0(\MHM(\Msp^{ss}_\mu))[\LL^{-1/2}, (\LL^r-1)^{-1}: r\ge 1]$, using the shorthand $[\PP^{f\cdot d-1}]_{vir}:= \frac{\LL^{f\cdot d/2} - \LL^{-f\cdot d/2}}{\LL^{1/2}-\LL^{-1/2}}$. Here $\pi:\Msp^{ss}_{f,\mu} \longrightarrow\Msp^{ss}_\mu$ is the morphism forgetting the framing. 
\end{proposition}
The parity assumption on the framing vector is made to avoid typical ``sign problems''. 

\subsection{The main result}

We also need  the following result proven in section 6.
\begin{theorem} \label{intconj}
If $\zeta$ is $\mu$-generic and $i_x:\Spec\CC\hookrightarrow \Msp_\mu$ the embedding corresponding to an arbitrary closed point  $x\in \Msp$, then the ``value'' $\DTS(x):=i_x^\ast\DTS$ of the Donaldson--Thomas function $\DTS$ is in the image of the natural map \[\Ka_0(\MHM(\CC))[\LL^{-1/2}] \longrightarrow \Ka_0(\MHM(\CC))[\LL^{-1/2}, (\LL^r-1)^{-1}: r\ge 1].\]  
\end{theorem}
\begin{remark} \rm Note that $\Ka_0(\MHM(\Msp_d^{ss}))$ is free over $\ZZ[\LL^{\pm 1}]$. Indeed, the set of all intersection complexes $\ICS_Z(L)$, with $Z$ running through all irreducible closed subvarieties of $\Msp_d^{ss}$ and $L$ running through equivalence classes of all irreducible,  polarizable, admissible variations of pure Hodge structures $L$ supported on $Z^o\subset Z_{reg}$ with quasi-unipotent monodromy at ``infinity'' and weight zero or one, provides a basis of the $\ZZ[\LL^{\pm 1}]$-module $\Ka_0(\MHM(\Msp_d^{ss}))$. As  $\ZZ[\LL^{\pm 1/2}] \hookrightarrow \ZZ[\LL^{-1/2}, (\LL^r-1)^{-1}: r\ge 1]$ is injective, we can, therefore, identify $\underline{\Ka}_0(\MHM(\Msp^{ss}_\mu))[\LL^{-1/2}] $ with a $\lambda$-subring of $\underline{\Ka}_0(\MHM(\Msp^{ss}_\mu))[\LL^{-1/2}, (\LL^r-1)^{-1}: r\ge 1]$ and similarly for $\Ka_0(\MHM(\CC))[\LL^{-1/2}]$.
\end{remark}
\begin{theorem} \label{DT=IC}
 Assume that $\zeta$ is $\mu$-generic. Then \[ \DTS_\mu=\ICS_{\overline{\Msp_\mu^{st}}}(\QQ) \] holds in $\underline{\Ka}_0(\MHM(\Msp^{ss}_\mu))[\LL^{-1/2}]$. In particular, for generic $\zeta$ 
 \[ \DT_d=\begin{cases} \IC_c(\Msp^{ss}_d,\QQ)=\IC(\Msp^{ss}_d,\QQ)^\vee &\mbox{ if } \Msp^{st}_d\neq \emptyset, \\
          0 & \mbox{ otherwise}
         \end{cases} \]
holds in $\Ka_0(\MHM(\CC))[\LL^{-1/2}]$ for every dimension vectors $d\in \Lambda_\mu$.
\end{theorem}
\begin{proof} As we have already mentioned,  $\Ka_0(\MHM(\Msp^{ss}_d))[\LL^{-1/2}]$ is a free $\ZZ[\LL^{\pm 1/2}]$ module with a basis given by the classes of $\ICS_Z(L)$. Here, $Z$ is an irreducible closed subvariety of $\Msp^{ss}_d$ and $L$ a pure irreducible, polarizable, admissible
variation of pure Hodge structures on $Z^o\subset Z_{reg}$ of weight zero with quasi-unipotent monodromy at ``infinity''. Two pairs $(Z,L)$ and $(Z',L')$ define the same intersection complex if $Z=Z'$ and $L|_{Z^o\cap Z'^o}=L'|_{Z^o\cap Z'^o}$. We get $i_x^\ast \ICS_Z(L)=\LL^{-\dim Z/2} \ICS_x(L_x)$ in $\Ka_0(\MHM(\CC)[\LL^{- 1/2}]$ with $L_x:=i^\ast_xL$ for a generic complex point $x\in Z$.\\ 
We prove the theorem by induction over $|d|$ starting with the trivial case $d=0$ for which the theorem is obviously true as $\Msp^{st}_0=\emptyset$. 
As before, $\Part$ denotes the set of all partitions of arbitrary size and $\Lambda'_\mu=\Lambda_\mu\setminus \{0\}$. We fix a framing vector $f\in \NN^{Q_0}$ such that $2|f_i$ for all $i\in Q_0$ and rewrite equation (\ref{eqn4}) using equations (\ref{eqn1}) and (\ref{eqn2}):
\[ \pi_\ast \ICS_{\Msp^{ss}_{f,d}}=\sum_{\begin{array}{c} \scriptstyle\lambda: \Lambda'_\mu \rightarrow \Part \\ \scriptstyle \sum |\lambda_e|e=d\end{array}} \prod_{e\in \Lambda'_\mu} S^{\lambda_e}[\PP^{f\cdot e -1}]_{vir} \cdot S^{\lambda_e}\DTS_e.\]
By induction over $|d|=\sum_{i\in Q_0}d_i$, we conclude using equation (\ref{eqn3}) that
\begin{eqnarray}
 \pi_\ast \ICS_{\Msp^{ss}_{f,d}}& =& \underbrace{[\PP^{f\cdot d-1}]_{vir}\DTS_d}_{\mbox{\scriptsize for }\lambda=\delta_d} + \sum_{\begin{array}{c} \scriptstyle \lambda: \Lambda'_\mu \rightarrow \Part \\ \scriptstyle \sum |\lambda_e|e=d \\ \scriptstyle \lambda\neq\delta_d \end{array}} \Big( \prod_{e\in \Lambda'_\mu} S^{\lambda_e}[\PP^{f\cdot e -1}]_{vir} \Big) \ICS_{Z_\lambda}(L_\lambda) \nonumber \\
 \label{eqn5} & = & \frac{\LL^{fd/2}-\LL^{-fd/2}}{\LL^{1/2}-\LL^{-1/2}}\DTS_d + \sum_{\begin{array}{c} \scriptstyle \lambda: \Lambda'_\mu \rightarrow \Part \\ \scriptstyle \sum |\lambda_e|e=d \\ \scriptstyle \lambda\neq\delta_d \end{array}} h_\lambda(\LL^{1/2})\cdot \ICS_{Z_\lambda}(L_\lambda),
\end{eqnarray}
for some palindromic Laurent polynomials $h_\lambda(\LL^{1/2})=h_\lambda(\LL^{-1/2})$ of degree at most $f\cdot d-\sum_e |\lambda_e|<f\cdot d -1$, some irreducible closed subvarieties $Z_\lambda$ and some variations $L_\lambda$ of Hodge structures of weight zero.\\
On the other hand, we can use the 
fact that $\pi$ is virtually small (see Theorem \ref{virtsmall}) and the relative hard Lefschetz theorem applied to the projective morphism $\pi$ to conclude
\begin{equation} \label{eqn6}
\pi_\ast \ICS_{\Msp^{ss}_{f,d}}=[\PP^{f\cdot d-1}]_{vir}\ICS_{\overline{\Msp^{st}_d}}(\QQ) + \sum_{(Z,L),\,Z\not=\overline{\Msp^{st}}} g_{Z,L}(\LL^{1/2}) \:\ICS_Z(L) 
\end{equation}
for certain palindromic Laurent polynomials $g_{Z,L}(\LL^{1/2})=g_{Z,L}(\LL^{-1/2})$ of degree less than $f\cdot d-1$. Here, $\ICS_{\overline{\Msp^{st}_d}}(\QQ)$ is zero if $\Msp^{st}_d=\emptyset$. Combining both equations, we get
\[ \frac{\LL^{fd/2}-\LL^{-fd/2}}{\LL^{1/2}-\LL^{-1/2}}\Big(\DTS_d-\ICS_{\overline{\Msp^{st}_d}}(\QQ)\Big)= \sum_{(Z,L), \,Z\not=\overline{\Msp^{st}_d}} f_{Z,L}(\LL^{1/2})\ICS_Z(L) \]
for certain palindromic Laurent polynomials $f_{Z,L}(\LL^{1/2})=f_{Z,L}(\LL^{-1/2})$ of degree less than $f\cdot d-1$. The sum on the right hand side is taken over pairs $(Z,L)$ as above (up to equivalence). We claim that both sides of the equation are zero. If not, we pick among all pairs $(Z,L)$ with $f_{Z,L}\not=0$ one for which $Z$ is of maximal dimension. Hence, we can find a closed point $x\in Z^o$ not contained in any other $Z'$ with $f_{Z',L'}\not=0$. Using the notation $i_x:\Spec\CC\to \Msp^{ss}_d$, we get
\[ \frac{\LL^{fd/2}-\LL^{-fd/2}}{\LL^{1/2}-\LL^{-1/2}}\Big(\DTS(x)-i^\ast_x\ICS_{\overline{\Msp^{st}_d}}(\QQ)\Big)=\LL^{-\dim Z/2}f_{Z,L}(\LL^{1/2})\,\ICS_x(L_x) \]
which is now an equation in the free $\ZZ[\LL^{\pm 1/2}]$-module $\Ka_0(\MHM(\CC))[\LL^{-1/2}]$ due to Theorem \ref{intconj}. In particular, the coefficient in front of the basis vector $\ICS_x(\LL_x)$ on the right hand side of the equation must be divisible in $\ZZ[\LL^{\pm 1/2}]$ by the palindromic Laurent polynomial
\[ \frac{\LL^{fd/2}-\LL^{-fd/2}}{\LL^{1/2}-\LL^{-1/2}}=\LL^{\frac{fd-1}{2}}+ \ldots + \LL^{\frac{1-fd}{2}} \]
of degree $fd-1$ in $\LL^{1/2}$ which is impossible as the degree of $f_{Z,L}$ is strictly smaller. Thus, the claim is proven, and 
$\DTS_d=\ICS_{\overline{\Msp^{st}_d}}(\QQ)$ follows.
\end{proof}

\subsection{Application to matrix invariants}

Since the motivic DT invariants of $m$-loop quivers are computed explicitly in \cite{Reineke4}, our main result allows us to give an explicit formula for the Poincar\'e polynomial in (compactly supported) intersection cohomology of the corresponding moduli spaces, which are the classical spaces of matrix invariants.\\
So let $Q^{(m)}$ be the quiver with a single vertex and $m\geq 2$ loops (in the case of no loop, or of one loop, the non-empty moduli spaces reduce to affine spaces). We consider the trivial stability and a positive integer $d$, and fix an $d$-dimensional $\CC$-vector space $V$. Then the moduli space $\Msp_d^{ss}(Q^{(m)})$ equals the invariant theoretic quotient $\Msp_d^{(m)}:=\End_\CC(V)^m/\!\!/\Gl_\CC(V)$ of $m$-tuples of linear operators up to simultaneous conjugation. This is an irreducible normal affine variety of dimension $(m-1)d^2+1$, singular except in case $d=1$ or $m=d=2$.\\
To formulate the explicit formula for the compactly supported intersection Betti numbers of $\Msp_d^{(m)}$, we need some combinatorial notions from \cite{Reineke4}. Let $U_d$ be the set of sequences $(a_1,\ldots,a_d)$ of natural numbers summing up to $(m-1)d$, on which the cyclic group $C_d$ of order $d$ acts by cyclic permutation. We call a sequence $a_*$ primitive if it is different from all its cyclic permutation, and almost primitive if it is either primitive, or $m$ is even, $d\equiv 2\bmod 4$, and the sequence equals twice a primitive sequence of length $d/2$. We define the degree of the sequence as $\sum_{i=1}^d(d-i)a_i$ and  the degree of a cyclic class of sequences as the minimal degree of sequences in this class. Let $U_d^{ap}/C_d$ be the set of cyclic classes of almost primitive sequences. Combining our main result with the formula for DT invariants in \cite{Reineke4}, we arrive at:

\begin{theorem}\label{mi} For all $d\geq 1$ and $m\geq 2$, we have
$$\sum_p\dim\IC^p_c(\Msp_d^{(m)},\QQ)v^p=v^{(m-1)d^2+1}\frac{1-v^{-2}}{1-v^{-2d}}\sum_{C\in U_d^{ap}/C_d}v^{-2\deg C}.$$
\end{theorem}

\section{Proof of Theorem \ref{virtsmall}}

\subsection{The stack of nilpotent quiver representations}

As before, let $Q$ be a finite quiver and $d\in\NN^{Q_0}$ a dimension vector for $Q$. Consider the action of the linear algebraic group $G_d$ on the vector space $R_d$. Let $p:R_d\rightarrow R_d/\!\!/G_d$ be the invariant-theoretic quotient; in other words, $R_d/\!\!/G_d$ is the spectrum of the ring of $G_d$-invariants in $R_d$, which, by \cite{LeBruyn-Procesi}, is generated by traces along oriented cycles in $Q$. We consider the nullcone of the representation of $G_d$ on $R_d$, that is,
$$N_d:=p^{-1}(p(0)).$$

By a standard application of the Hilbert criterion (see \cite[Chapter 6]{LeBruyn} for a much finer analysis of the geometry of $N_d$ using the Hesselink stratification), we can characterize points in $N_d$ either as those representations such that every cycle is represented by a nilpotent operator, or as those representation admitting a composition series by the one-dimensional irreducible representations $S_i$ concentrated at a single vertex $i\in Q_0$ (and with all loops at $i$ represented by $0$).\\
The main observation of this section is that, under the assumption of $Q$ being symmetric, there is an effective estimate for the dimension of $N_d$.

\begin{theorem}\label{nilpotent stack dimension} If $Q$ is symmetric, we have
$$\dim N_d-\dim G_d\leq -\frac{1}{2}( d,d)+\frac{1}{2}\sum_{i\in Q_0}( i,i) d_i-|d|.$$
\end{theorem}

\begin{proof} For a decomposition $d=d^1+\ldots+d^s$, denoted with $d^*$, we consider the closed subvariety $R_{d^*}$ of $R_d$ consisting of representations $V$ admitting a filtration $0=V_0\subset V_1\subset \ldots\subset V_s=V$ by subrepresentations, such that $V_k/V_{k-1}$ equals the zero representation of dimension vector $d^k$ for all $k=1,\ldots,s$. This subvariety being the collapsing of a homogeneous bundle over a variety of partial flags in $\bigoplus_{i\in Q_0}\KK^{d_i}$, its dimension is easily estimated as
$$\dim R_{d^*}\leq \dim G_d-\sum_{k<l}( d^l,d^k)-\sum_{i\in Q_0}\sum_k(d^k_i)^2.$$
The above characterization of $N_d$ allows us to write $N_d$ as the union of all $R_{d^*}$ for decompositions $d^*$ which are thin, that is, all of whose parts are one-dimensional (one-dimensionality is obscured by the notation to avoid multiple indexing and to make the argument more transparent). Thus $\dim N_d-\dim G_d$ is bounded from above by the maximum of the values
$$-\sum_{k<l}( d^l,d^k)-\sum_{i\in Q_0}\sum_k(d^k_i)^2$$
over all thin decompositions. Since $Q$ is symmetric, we can rewrite
$$\sum_{k<l}( d^l,d^k)=\frac{1}{2}( d,d)-\frac{1}{2}\sum_k( d^k,d^k).$$
All $d^k$ being one-dimensional, we can easily rewrite
$$\sum_{i\in Q_0}\sum_k(d^k_i)^2=|d|,\;\;\;
\sum_k( d^k,d^k)=\sum_{i\in Q_0}( i,i) d_i.$$
All terms now being independent of the chosen thin decomposition, we arrive at the required estimate.
\end{proof}

\subsection{Virtual smallness of the Hilbert--Chow map}

We consider again the Hilbert--Chow map $\pi:\Msp_{f,d}^{ss}\rightarrow\Msp_d^{ss}$ forgetting the framing datum; our aim is to prove a strong dimension estimate for its fibers when the stability is $\mu$-generic (cf.\ section 2.2) for $\mu$ being the slope of $d$.\\
We consider the Luna stratification of $\Msp_d^{ss}$: a decomposition type $\xi$ for $d$ consists of a sequence $((d^1,m_1),\ldots,(d^s,m_s))$ in $\Lambda_\mu\times\NN$ such that $\sum_km_kd^k=d$. Inside the moduli space $\Msp_d^{ss}$ parameterizing isomorphism classes of polystable representations of dimension vector $d$, we can consider the subset $S_\xi$ of representations of the form $\bigoplus_k E_k^{m_k}$ for pairwise non-isomorphic stable representations $E_k$ of dimension vector $d^k$ and slope $\mu$. We thus have
$$\dim S_\xi=\sum_k\dim \Msp^{st}_{d^k}(Q)=s-\sum_k( d^k,d^k).$$
By \cite{Reineke1}, $S_\xi$ is locally closed, and the map $\pi$ is \'etale locally trivial over $S_\xi$. We fix a point $x\in S_\xi$. This stratum being nonempty, $\Msp_{d^k}^{st}(Q)$ is nonempty, and thus $( d^k,d^k)=1-\dim \Msp_{d^k}^{st}(Q)\leq 1$ for all $k$. The fiber $\pi^{-1}(x)$ over a point $x\in S_\xi$ can be described as follows:\\
Define the local quiver $Q_\xi$ with vertices $i_1,\ldots,i_s$ and $\delta_{kl}-( d^k,d^l)$ arrows from $i_k$ to $i_l$. Define a local dimension vector $d_\xi$ for $Q_\xi$ by $(d_\xi)_{i_k}=m_k$, and a local framing datum $f_\xi$ by $(f_\xi)_{i_k}=f\cdot d^k$. We consider the trivial stability on $Q_\xi$. Then we have a local Hilbert--Chow map $$\pi_\xi:\Msp_{f_\xi,d_\xi}^{ssimp}(Q_\xi)\rightarrow \Msp_{d_\xi}^{ssimp}(Q_\xi)=R_{d_\xi}/\!\!/G_{d_\xi}.$$
We denote the fiber over the class of the zero representation by $M_{f_\xi,d_\xi}^{nilp}(Q_\xi).$ Then, by \cite{Reineke1}, we have
$$\pi^{-1}(x)\simeq M_{f_\xi,d_\xi}^{nilp}(Q_\xi).$$
By construction, we have
$$\dim M_{f_\xi,d_\xi}^{nilp}(Q_\xi)=\dim N_{d_\xi}-\dim G_{d_\xi}+f_\xi\cdot d_\xi.$$
Now assume $\zeta$ to be $\mu$-generic, thus $Q_\xi$ is symmetric, and Theorem \ref{nilpotent stack dimension} estimates the dimension of the fiber $\pi^{-1}(x)$ as
$$\dim\pi^{-1}(x)=\dim\Msp_{f_\xi,d_\xi}^{nilp}(Q_\xi)=\dim N_{d_\xi}-\dim G_{d_\xi}+f_\xi\cdot d_\xi\leq$$
$$\leq -\frac{1}{2}( d_\xi,d_\xi)_{Q_\xi}+\frac{1}{2}\sum_k( i_k,i_k)_{Q_\xi}(d_\xi)_{i_k}-|d_\xi| + f_\xi\cdot d_\xi.$$
Using the definition of $Q_\xi$, $d_\xi$ and $f_\xi$, this simplifies to

$$\dim\pi^{-1}(x)\leq -\frac{1}{2}( d,d)+\frac{1}{2}( d^k,d^k) m_k-\sum_km_k+f\cdot d.$$
On the other hand, we can rewrite the dimension formula for $S_\xi$ as
$${\rm codim} S_\xi=-( d,d)+\sum_k( d^k,d^k)+1-s.$$
The inequality
$$\dim \pi^{-1}(x)-(f\cdot d-1)\leq\frac{1}{2}{\rm codim} S_\xi$$
(with equality only if $0:=\xi=((d,1))$) claimed in Theorem \ref{virtsmall} can thus be rewritten as
$$-\frac{1}{2}( d,d)+\frac{1}{2}\sum_k( d^k,d^k) m_k-\sum_km_k+1\leq -\frac{1}{2}( d,d)+\frac{1}{2}\sum_k( d^k,d^k)+\frac{1}{2}(1-s).$$
This is easily simplified to
$$\frac{1}{2}\sum_k(( d^k,d^k)-2)(m_k-1)\leq\frac{1}{2}(s-1).$$
Since $( d^k,d^k)\leq 1$, the left hand side is nonpositive, whereas the right hand side is nonnegative. Equality holds if both sides are zero, thus $s=1$, proving virtual smallness.

\section{Motivic DT-theory and the integrality conjecture}

We prove a stronger version of Theorem \ref{intconj} for arbitrary ground fields $\kk$ with $\Char(\kk)=0$  and not necessarily closed $\kk$-points. Since it is not  clear how to deal with mixed Hodge modules on varieties defined over arbitrary  fields, we will work in the motivic world using motivic functions instead of mixed Hodge modules. The reader not familiar with motivic functions might have a look at \cite{JoyceMF}, where motivic functions are called stack functions. However, we will also recall the main definitions below.  The machinery used to define Donaldson--Thomas functions will also work in this more general context, and we prove a couple of useful formulas.  There is a $\lambda$-ring homomorphism from
\[\underline{\Ka}_0(\Var/\Msp^{ss})[\LL^{-1/2}, (\LL^r-1)^{-1}: r\ge 1] \]
to
\[ \underline{\Ka}_0(\MHM(\Msp^{ss}))[\LL^{-1/2}, (\LL^r-1)^{-1}: r\ge 1],\] induced by $[X\xrightarrow{q} \Msp]\mapsto q_! \QQ$, giving rise to corresponding results for mixed Hodge modules. As we will discuss at the end of this section, working with motivic functions has also some limitations.

\subsection{Motivic functions}

Given an arbitrary Artin stack or scheme $\BB$ with connected components being of finite type over\footnote{In practice, $\KK$ will be our ground field $\kk$ or some extension of $\kk$.} $\KK$, we define the Grothendieck group $\Ka_0(\Var/\BB)$ to be the free abelian group generated by isomorphism classes $[\mathcal{X} \rightarrow \BB]$ of representable morphisms of finite type such that  $\mathcal{X}$ has a locally finite stratification by quotient stacks $\mathcal{X}_i=X_i/\Gl_\KK(n_i)$, subject to the cut and paste relation 
\[ [\mathcal{X} \rightarrow \BB] = [\mathcal{Z} \rightarrow \BB]+ [\mathcal{X}\setminus \mathcal{Z} \rightarrow \BB], \]
for every closed substack $\mathcal{Z}\subset \mathcal{X}$. In particular, $[\mathcal{X}\to \BB]=[\mathcal{X}_{red}\to \BB]$.
\begin{remark} \rm \label{reduction_to_affine_case}
Using the cut and paste relation we arrive at the conclusion that if $\BB=\Spec B$ as an affine scheme of a finitely generated $\KK$-algebra $B$, the group $\Ka_0(\Var/\Spec B)$ can also be described as the  abelian group generated by symbols $[A]$ for each finitely generated $B$-algebras $A$ subject to the following two conditions.
\begin{enumerate}
\item If $A\cong A'$ as $B$-algebras, then $[A]=[A']$.
\item If $a_1,\ldots,a_r\in A$ is a finite set of elements, then 
\[ [A]=[A/(a_1,\ldots a_r)] \quad + \sum_{\emptyset\not= J\subset \{1,\ldots,r\}} (-1)^{|J|-1} [A_{\prod_{j\in J} a_j}].\]
\end{enumerate}
\end{remark}
The fiber product over $\KK$ defines a ring structure on $\Ka_0(\Var/\KK)$ and a $\Ka_0(\Var/\KK)$-module structure on $\Ka_0(\Var/\BB)$. Taking the product over $\KK$ defines an exterior product $\boxtimes:\Ka_0(\Var/\BB)\times \Ka_0(\Var/\BB')\longrightarrow \Ka_0(\Var/\BB\times_\KK \BB')$. Let us also introduce the module 
\[ \Ka_0(\Var/\BB)[\LL^{-1/2}, (\LL^r-1)^{-1}: r\ge 1]:=\Ka_0(\Var/\BB)\otimes_{\ZZ[\LL]} \ZZ[\LL^{-1/2}, (\LL^r-1)^{-1} : r\ge 1]\] 
with $\LL$ denoting the Lefschetz motive $\LL:=[\mathbb{A}^1_\KK]\in \Ka_0(\Var/\KK)$.\footnote{For $\BB=\Spec \KK$, we simplify the notation by suppressing the structure morphism to $\Spec \KK$.}  We will also add the relations 
\begin{equation} \label{principal_bundle_relation} [X/\Gl_\KK(n) \rightarrow \BB]= [X\rightarrow \BB]/[\Gl_\KK(n)] 
\end{equation}
for every $\Gl_n$-action on a scheme $X$. Here, $[\Gl_\KK(n)]=\LL^{n \choose 2}\prod_{r=1}^n (\LL^r-1)$. In particular, due to our assumption on $\mathcal{X}$ for a generator $[\mathcal{X}\to \BB]$, the group $\Ka_0(\Var/\BB)[\LL^{-1/2}, (\LL^r-1)^{-1}: r\ge 1]$ is generated as a $\ZZ[\LL^{-1/2}, (\LL^r-1)^{-1}: r\ge 1]$-module by morphisms $[X\rightarrow \BB]$, with $X$ being a scheme. Because of this and  Lemma 3.9 in \cite{Bridgeland10} which easily generalizes to the relative situation, the proper push forward $\phi_!$ along morphisms $\phi:\BB \rightarrow \BB'$ such that $\pi_0(\phi):\pi_0(\BB)\to \pi_0(\BB')$ has finite fibers is well defined by composition $\phi_!([X\rightarrow \BB])=[X\rightarrow \BB']$. \\
We can also define \[\phi^\ast:\Ka_0(\Var/\BB')[\LL^{-1/2},(\LL^r-1)^{-1}:r\ge 1] \longrightarrow  \Ka_0(\Var/\BB)[\LL^{-1/2},(\LL^r-1)^{-1}:r\ge 1] \]
for all $\phi:\BB\to \BB'$ via $\phi^\ast([\mathcal{X}\to \BB'])=[\mathcal{X}\times_{\BB'}\BB \to \BB]$ on generators. This definition makes even sense if $\BB$ and $\BB'$ are defined over different ground fields $\KK$ and $\KK'$. We will also introduce the group
\begin{eqnarray*} \lefteqn{\underline{\Ka}_0(\Var/\BB)[\LL^{-1/2}, (\LL^r-1)^{-1}: r\ge 1]} & & \\ &=&\prod_{\BB_i\in \pi_0(\BB)} \Big(\Ka_0(\Var/\BB_i)[\LL^{-1/2}, (\LL^r-1)^{-1}: r\ge 1]\Big). \end{eqnarray*}
The pull-back and the push-forward satisfy some base change formula for every cartesian square. Moreover, for every quotient stack $\rho:X\to X/G$ with $G$ being a special linear algebraic group, the formula 
\begin{equation} \label{push_pull} \rho_!\rho^\ast(f)=[G]\cdot f
\end{equation}
holds for all $f\in \underline{\Ka}_0(\Var/\BB)[\LL^{-1/2},(\LL^r-1)^{-1}:r\ge 1]$, and $[G]$ is invertible. Indeed, if $[Y\xrightarrow{u} X/G]$ is a generator, then $\rho_!\rho^\ast[Y\to X/G]=[Y\times_{X/G} X \longrightarrow Y \longrightarrow X/G]$ with $P=Y\times_{X/G}X$  being  a principal $G$-bundle on $Y$. As $G$ is special, $P\to Y$  is Zariski locally trivial, and $[P\to Y]=[G][Y\to Y]$ follows in $\underline{\Ka}_0(\Var/Y)[\LL^{-1/2},(\LL^r-1)^{-1}:r\ge 1]$. Hence, 
\[\rho_!\rho^\ast([Y\to X/G])=[P\to Y\xrightarrow{u} X/G]= u_!([P\to Y])=[G][Y\to X/G]. \] 
The principal $G$-bundle $\Gl_\KK(n)\to \Gl_\KK(n)/G$ is Zariski locally trivial and $[\Gl_\KK(n)]=[G][\Gl_\KK(n)/G]$ is invertible proving the invertibility of $[G]$.

\subsection{$\lambda$-ring structures}

If the base $\BB$ is a scheme and has an additional structure of a commutative monoid with zero $\Spec \KK \xrightarrow{\;0\;}\BB$ and sum $\oplus:\BB\times_\KK \BB \rightarrow \BB$, then $\Ka_0(\Var/\BB)$ can be equipped with the structure of a $\lambda$-ring by putting
\begin{eqnarray*} [X\rightarrow \BB]\cdot [Y\rightarrow \BB] &:=& [X\times_\KK Y \rightarrow \BB\times_\KK \BB \xrightarrow{\;\oplus\;} \BB] \;\mbox{ and} \\
\sigma^n([X \rightarrow \BB]) &:=& [ \Sym_\KK^n(X) \rightarrow \Sym_\KK^n(\BB) \xrightarrow{\;\oplus\;} \BB] \end{eqnarray*}
with $\Sym^n_\KK(X)=X^{\times_\KK n}/\!\!/S_n$. On can extend the $\lambda$-ring structure to $\Ka_0(\Var/\BB)[\LL^{-1/2}, (\LL^r-1)^{-1}: r\ge 1]$ by defining $-\LL^{1/2}$ to be a line element, that is, $\sigma^n(-\LL^{1/2}):=(-\LL^{1/2})^n$. Moreover, the $\lambda$-ring structure extends to $\underline{\Ka}_0(\Var/\BB)[\LL^{-1/2}, (\LL^r-1)^{-1}: r\ge 1]$.\\
Given a motivic function $f\in \underline{\Ka}_0(\Var/\BB)[\LL^{-1/2}, (\LL^r-1)^{-1}: r\ge 1]$ such that $\sigma^n(f)|_{\BB_i}$ vanishes for all but finitely many $n\in \NN$ depending on the connected component $\BB_i$ of $\BB$, the sum 
\[ \Sym(f):=\sum_{n\ge 0} \sigma^n(f) \]
is well defined in $\underline{\Ka}_0(\Var/\BB)[\LL^{-1/2}, (\LL^r-1)^{-1}: r\ge 1]$ and satisfies $\Sym(0)=1=[\Spec \KK \xrightarrow{\;0\;}\BB]$ as well as $\Sym(f+g)=\Sym(f)\cdot\Sym(g)$.\\

Formation of (direct) sums of semisimple objects in $\KK Q\rep$ and dimension vectors in $\NN^{Q_0}$, provides $\Msp^{ss}_\mu$ and $\NN^{Q_0}\times\Spec\KK$ with the structure of a commutative monoid inducing a $\lambda$-ring structure on $\underline{\Ka}_0(\Var/\Msp^{ss}_\mu)[\LL^{-1/2}, (\LL^r-1)^{-1}: r\ge 1]$ and on $\underline{\Ka}_0(\Var/\NN^{Q_0}\times \Spec\KK)[\LL^{-1/2}, (\LL^r-1)^{-1}: r\ge 1]$. Notice that the latter $\lambda$-ring is isomorphic to the $\lambda$-ring 
$\Ka_0(\Var/\KK)[\LL^{-1/2}, (\LL^r-1)^{-1}: r\ge 1][[t_i:i\in Q_0]]$
of power series.
If a motivic function $f$ on $\Msp^{ss}_\mu$, respectively on $\Lambda_\mu\times\Spec\KK$, is supported away from the zero representation, the infinite sum $\Sym(f)$ is 
well defined.

\begin{lemma} \label{lambda_pull_back} 
Let $M$ and $N$ be commutative monoids in the category of schemes over fields $\kk$ and $\KK\supset \kk$ respectively of characteristic zero. Assume that $\iota:N\to M$ induces a homomorphism $N\longrightarrow M\otimes_\kk \Spec\KK$ (over $\KK$) of commutative monoids over $\KK$ such that the map $u_n$ in the  diagram
\[
\xymatrix @C=2cm{ N^{\times_{\KK} n} \ar[r]^(0.4){u_n}  \ar[dr]_{\oplus} & N\times_M M^{\times_{\kk} n} \ar[r] \ar[d] & M^{\times_\kk n} \ar[d]^\oplus \\ & N \ar[r]^\iota & M }
\]
is a closed embedding and an isomorphism between geometric points for every $n\in \NN$. Then $\iota^\ast(fg)=\iota^\ast(f)\iota^\ast(g)$ and $\iota^\ast(\sigma^n(f))=\sigma^n(\iota^\ast(f))$ for all $n\in \NN$ and all $f,g\in \underline{\Ka}_0(\Var/M)[\LL^{-1/2},(\LL^r-1)^{-1}: r\ge 1]$.
\end{lemma}
\begin{proof} We will show $\iota^\ast(\sigma^n(f))=\sigma^n(\iota^\ast(f))$ for a generator $[X\to M]$ and leave the rest to the reader. By definition, $\iota^\ast([X\to M])=[Y\to N]$ using the shorthand $Y:=N\times_M X$. By the properties of $u_n$,  the map $u'_n$ in the diagram 
\[
\xymatrix @C=2cm { Y^{\times_\KK n} \ar[d] \ar[r]^{u'_n} &  N\times_M X^{\times_\kk n} \ar[r] \ar[d] & X^{\times_\kk n} \ar[d] \\ \Sym^n_\KK(Y) \ar[dr] \ar[r]^{u''_n} & N\times_M \Sym_\kk^n(X) \ar[d] \ar[r] & \Sym_\kk^n(X) \ar[d] \\ & N \ar[r]^\iota   & M }
\]
is also a closed embedding inducing an isomorphism between geometric points. By general GIT-theory, $N\times_M \Sym_\kk^n(X)$ is the categorical quotient of $N\times_M X^{\times_\kk n}$ with respect to the induced $S_n$-action. It can be computed Zariski locally by taking $S_n$-invariant functions. As $\Char(\KK)=0$, $S_n$ acts linearly reductive on $\KK$-vector spaces, and the map $u''_n$ must also be a closed embedding. Since $u'$ induces a bijection between geometric points, the same must hold for $u''_n$ and $\Sym^n(Y)_{red}\cong (N\times_M\Sym_\kk^n(X))_{red}$ follows.  Thus, $\iota^\ast([\Sym^n_\kk(X)\to M])=[N\times_M\Sym^n_\kk(X)\longrightarrow N]=[\Sym^n_\KK(Y)\longrightarrow N]$.
\end{proof}

\subsection{Convolution product and integration map}

Throughout the next three subsections, all schemes and stacks are defined over a  field $\KK$ which might be an extension of  another fixed ground field $\kk$. Unless otherwise stated, cartesian products are taken over $\Spec\KK$.
We define a ``convolution'' product, the so-called Ringel--Hall product, on $\underline{\Ka}_0(\Var/\Mst^{ss}_\mu)[\LL^{-1/2},(\LL^r-1)^{-1}:r\ge 1]$ by means of the following diagram
\[ \xymatrix { & {\EE xact^{ss}_\mu} \ar[dr]^{\pi_2} \ar[dl]_{\pi_1\times\pi_3} & \\ \Mst^{ss}_\mu \times \Mst^{ss}_\mu & & \Mst^{ss}_\mu } \]
via $f\ast g :=\pi_{2\, !}(\pi_1\times \pi_3)^\ast(f\boxtimes g)$, where $\EE xact^{ss}_\mu$ denotes the stack of short exact sequences $0\to V_1 \to V_2 \to V_3 \to 0$ of semistable representations of slope $\mu$, and $\pi_i$ maps such a sequence to its $i$-th entry. It is well-known that the convolution product provides $\underline{\Ka}_0(\Var/\Mst^{ss}_\mu)[\LL^{-1/2},(\LL^r-1)^{-1}:r\ge 1]$ with a $\Ka_0(\Var/\kk)[\LL^{-1/2},(\LL^r-1)^{-1}:r\ge 1]$-algebra structure with unit given by the motivic function $[\Spec\kk \xrightarrow{\;0\;} \Mst^{ss}_\mu]$.

\begin{lemma} \label{integration_map}
The ``integration'' map 
\[I^{ss}_\mu:\underline{\Ka}_0(\Var/\Mst^{ss}_\mu)[\LL^{-1/2},(\LL^r-1)^{-1}:r\ge 1] \longrightarrow \underline{\Ka}_0(\Var/\Msp^{ss}_\mu)[\LL^{-1/2},(\LL^r-1)^{-1}:r\ge 1] \]
given by $I^{ss}_\mu(f):=\sum_{d\in \Lambda_\mu} \LL^{(d,d)/2} p_{d\, !}(f|_{\Mst^{ss}_d})$ is a $\Ka_0(\Var/\kk)[\LL^{-1/2},(\LL^r-1)^{-1}:r\ge 1]$-algebra homomorphism with respect to the convolution product if $\zeta$ is $\mu$-generic. 
\end{lemma}
\begin{proof}
We use the notation of the following commutative diagram. 
\[
 \xymatrix { & X^{ss}_{d,d'} \ar[dl]_{\hat{\pi}_1\times\hat{\pi}_3} \ar@{^{(}->}[dr]^{\hat{\pi}_2} \ar[dd]^(0.3){\rho_{d,d'}}& \\
 X^{ss}_d\times X^{ss}_{d'} \ar[dd]_{\rho_d\times \rho_{d'}} & & X^{ss}_{d+d'} \ar[dd]^{\rho_{d+d'}}  \\
 & X^{ss}_{d,d'}/G_{d,d'} \ar[dl]^{\pi_1\times \pi_3} \ar[dr]_{\pi_2} & \\
 X^{ss}_d/G_d \times X^{ss}_{d'}/G_{d'} \ar[dd]_{p_d\times p_{d'}} & & X^{ss}_{d+d'}/G_{d+d'} \ar[dd]^{p_{d+d'}}\\ & & \\
 \Msp^{ss}_d\times \Msp^{ss}_{d'} \ar[rr]^\oplus & & \Msp^{ss}_{d+d'} }
\]

The first computation generalizes formula (\ref{push_pull}) to the map $\pi_1\times \pi_3$ by applying (\ref{push_pull}) to the principal bundles $\rho_{d,d'}$, $\hat{\pi}_1\times \hat{\pi}_3$ and $\rho_d\times \rho_{d'}$ with special linear structure groups \[G_{d,d'}, \quad\bigoplus_{Q_1\ni \alpha:i\to j} \Hom_\kk(\kk^{d'_i},\kk^{d_j}) \quad\mbox{and}\quad G_d\times G_{d'}.\]
For $h\in \underline{\Ka}_0(\Var/\Mst^{ss}_d\times\Mst^{ss}_{d'})[\LL^{-1/2},(\LL^r-1)^{-1}:r\ge 1]$ we get
\begin{eqnarray*}
(\pi_1\times \pi_3)_!(\pi_1\times \pi_3)^\ast(h)&=& \frac{1}{[G_{d,d'}]}(\pi_1\times \pi_3)_!\rho_{d,d' \,!}\rho_{d,d'}^\ast (\pi_1\times \pi_3)^\ast(h) \\
&=& \frac{1}{[G_{d,d'}]}(\rho_d\times \rho_{d'})_!(\hat{\pi}_1\times\hat{\pi}_3)_!((\hat{\pi}_1\times\hat{\pi}_3)^\ast(\rho_d\times \rho_{d'})^\ast(h) \\
&=& \frac{\LL^{dd'-(d',d)}}{[G_{d,d'}]}(\rho_d\times \rho_{d'})_!(\rho_d\times \rho_{d'})^\ast(h)\\
&=& \LL^{-(d',d)} h.
\end{eqnarray*}
Thus, for $f\in  \underline{\Ka}_0(\Var/\Mst^{ss}_d)[\LL^{-1/2},(\LL^r-1)^{-1}:r\ge 1]$ and $g\in  \underline{\Ka}_0(\Var/\Mst^{ss}_{d'})[\LL^{-1/2},(\LL^r-1)^{-1}:r\ge 1]$
\begin{eqnarray*}
I^{ss}_\mu(f\ast g) &=& \LL^{(d+d',d+d')/2} p_{(d+d')!}(f\ast g ) \\ 
&=& \LL^{(d,d)/2}\LL^{(d',d')/2} \LL^{(d',d)} (p_{d+d'}\pi_2)_!(\pi_1\times \pi_3)^\ast(f\boxtimes g)\\
&=& \LL^{(d,d)/2}\LL^{(d',d')/2} \LL^{(d',d)} \big(\oplus(p_d\times p_{d'})(\pi_1\times \pi_3)\big)_!(\pi_1\times \pi_3)^\ast(f\boxtimes g) \\
&=& \LL^{(d,d)/2}\LL^{(d',d')/2}  \oplus_!(p_d\times p_{d'})_!(f\boxtimes g)\\
&=& I^{ss}_{\mu}(f)\cdot I^{ss}_{\mu}(g).
\end{eqnarray*}

\end{proof}

\subsection{A useful identity}

Fix a framing vector $f\in \NN^{Q_0}$ and use the notation of Section 2. Consider the motivic functions $H:=[\Mst^{ss}_{f,\mu}\xrightarrow{\tilde{\pi}} \Mst^{ss}_\mu]$ and $\unit_{\Xst}:=[\Xst\xrightarrow{\,\id\,} \Xst]$  for any Artin stack $\Xst$. Then, 
\begin{equation} \label{Hilbert_scheme_identity} 
\Big(H\ast \unit_{\Mst_\mu^{ss}}\Big)|_{\Mst^{ss}_d}=\frac{\LL^{fd}}{\LL-1}\unit_{\Mst^{ss}_d}.
\end{equation} 
Indeed, consider the following commutative diagram
\[ \xymatrix @C=1.5cm{ {\Xst}:=\mathfrak{E}xact(Q_f)|_{\Mst^{ss}_{f,\mu}\times \Mst^{ss}_\mu} \ar[d]^{\pi_1^f\times \pi_3^f} \ar[r]^(0.6){\hat{\pi}} & 
\mathfrak{E}xact(Q)|_{\Mst^{ss}_\mu\times \Mst^{ss}_\mu} \ar[r]^(0.6){\pi_2} \ar[d]^{\pi_1\times\pi_3} & \Mst^{ss}_{\mu} \\ \Mst^{ss}_{f,\mu}\times\Mst^{ss}_{\mu} \ar[r]^{\tilde{\pi}\times\id_{\Mst^{ss}_\mu}} & \Mst^{ss}_\mu\times\Mst^{ss}_\mu, } \]
where the terms on the left hand side correspond to $Q_f$-representations with $\Mst^{ss}_\mu$ interpreted as the space of all $\zeta'$-semistable $Q_f$-representations with dimension vector in $\Lambda_\mu\times\{0\}$. The reader should convince himself that the square is cartesian and that $\Xst$ is the moduli stack of all $Q_f$-representations of dimension vector in $\Lambda_\mu\times\{1\}$ such that the restriction to the subquiver $Q$ is  $\zeta$-semistable. Indeed, any such representation $V$ has a unique semistable subrepresentation $V_c$ of the same slope  ``generated'' by $V_\infty\cong \kk$, i.e.\ a subrepresentation in $\Mst^{ss}_{f,\mu}$, and the quotient $V/V_c$ will be in $\Mst^{ss}_\mu$. By construction, $V_c|_Q$ is the intersection of all (semistable) subrepresentations $V'\subseteq V|_Q$ of slope $\mu$ containing all framing vectors. The map $\hat{\pi}$ restricts the short exact 
sequence $0\to V_c\to V\to V/V_c\to 0$ to $Q$. We finally get
\begin{eqnarray*} 
H\ast \unit_{\Mst_\mu^{ss}} &=& \pi_{2\, !} (\pi_1\times \pi_3)^\ast \big(\tilde{\pi}_{!}(\unit_{\Mst^{ss}_{f,\mu}}) \boxtimes \unit_{\Mst^{ss}_\mu} \big) \\
&=& \pi_{2\, !} (\pi_1\times \pi_3)^\ast (\tilde{\pi}\times \id_{\Mst^{ss}_\mu})_!\big(\unit_{\Mst^{ss}_{f,\mu}} \boxtimes \unit_{\Mst^{ss}_\mu}\big)  \\
&=& \pi_{2\,!}\hat{\pi}_! (\pi^f_1\times\pi^f_3)^\ast(\unit_{\Mst^{ss}_{f,\mu}\times\Mst^{ss}_\mu}) \\
&=& (\pi_{2}\hat{\pi})_! ( \unit_\Xst). 
\end{eqnarray*}
Looking at connected components, the map $\pi_2\hat{\pi}$ is a stratification of
\[ (X^{ss}_d\times \AA^{fd})/(G_d\times \GG_m) \xrightarrow{\tilde{\pi}_d} X^{ss}_d/G_d \]
with $\AA^{fd}_\kk$ parameterizing the matrix coefficients of the maps from $V_\infty\cong\kk$ to $V_i\cong\kk^{d_i}$ for $i\in Q_0$, i.e.\ the coordinates of the framing vectors, and $\GG_m$ corresponds to basis change in $V_\infty$. 
Applying equation  (\ref{push_pull}) to the principal $G_d$ respectively $G_d\times \GG_m$-bundles 
\begin{eqnarray*} 
X_d^{ss}&\xrightarrow{\;\rho_d\;}& X^{ss}_d/G_d, \\
\tilde{\pi}_d: X^{ss}_d\times \AA^{fd} &\xrightarrow{\;\tilde{\rho}_d\;}& X^{ss}_d\times \AA^{fd}/G_d\times \GG_m, \end{eqnarray*}
yields
\begin{eqnarray*}
\lefteqn{ \tilde{\pi}_{d\,!}\Big(\unit_{ X^{ss}_d\times \AA^{fd}/G_d\times \GG_m}\Big) }\\
&=& (\tilde{\pi}_d\circ \tilde{\rho_d})_!\Big(\unit_{X^{ss}_d\times\AA^{fd}}\Big)/[G_d\times\GG_m],\\
&=& (\rho_d\circ \pr_{X^{ss}_d})_!\Big(\unit_{X_d^{ss}\times \AA^{fd}}\Big)/[G_d\times\GG_m], \\
&=& \frac{\LL^{fd}}{\LL-1}\rho_{d\,!} \Big(\unit_{X_d^{ss}}\Big)/[G_d], \\
&=& \frac{\LL^{fd}}{\LL-1} \unit_{\Mst^{ss}_d}, 
\end{eqnarray*}
and the equation for the restriction of $H\ast \unit_{\Mst^{ss}_\mu}$ to $\Mst^{ss}_d$ follows.

\subsection{Donaldson--Thomas invariants}

The following definition of Donaldson--Thomas invariants is a simplified version of a  more general and much more complicated one which can be applied to triangulated 3-Calabi--Yau $A_\infty$-categories. We can embed $\kk Q\rep$ into the 3-Calabi--Yau $A_\infty$-category $D^b(\Gamma_\kk Q\rep)$ introduced in section 2.1, and the general version reduces to the one given here. \\

For $\mu\in (-\infty,+\infty]$ we define  the motivic version of the intersection complex $\ICS_{\Mst^{ss}_\mu}$  by the following motivic function in $\Mst^{st}_\mu$
\[ \ICS_{\Mst^{ss}_\mu}^{mot}:= \sum_{d\in \Lambda_\mu} \LL^{(d,d)/2}[\Mst^{ss}_d \hookrightarrow \Mst^{ss}_\mu]. \]
Here, $\LL^{(d,d)/2}$ is the analog of the normalization factor for mixed Hodge modules since $\dim \Mst^{ss}_d=-(d,d)$.
Taking the proper push forward along the morphisms $p:\Mst^{ss}_\mu \rightarrow \Msp^{ss}_\mu$ and $\dim:\Msp^{ss} \rightarrow \NN^{Q_0}\times\Spec\kk$ respectively, we can define the motivic Donaldson--Thomas function $\DTS^{mot}\in \underline{\Ka}_0(\Var/\Msp^{ss})[\LL^{-1/2}, (\LL^r-1)^{-1}: r\ge 1]$ and the generating series $\DT^{mot}:=\dim_!\DTS^{mot}\in \Ka_0(\Var/\kk)[\LL^{-1/2}, (\LL^r-1)^{-1}: r\ge 1][[t_i:i\in Q_0]]$ of the motivic Donaldson--Thomas invariants  by $\DTS^{mot}|_{\Msp^{ss}_\mu}=\DTS^{mot}_\mu$ for all $\mu\in (-\infty,+\infty]$ with $\DTS^{mot}_\mu$ being the unique solution of the equation
\[ p_! \ICS_{\Mst^{ss}_\mu}^{mot} = \Sym\Bigl( \frac{1}{\LL^{1/2}-\LL^{-1/2}}\, \DTS^{mot}_\mu \Bigr)\]
such that $\DTS^{mot}_\mu|_{\Msp^{ss}_0}=0$. As $\dim_!$ is a $\lambda$-ring homomorphism from the $\lambda$-ring $\underline{\Ka}_0(\Var/\Msp^{ss})[\LL^{-1/2}, (\LL^r-1)^{-1}: r\ge 1]$ to $ \Ka_0(\Var/\kk)[\LL^{-1/2}, (\LL^r-1)^{-1}: r\ge 1][[t_i:i\in Q_0]]$, this implies
\begin{eqnarray*}
\dim_! p_! \ICS^{mot}_{\Mst^{ss}_\mu} &=& \Sym\Bigl( \frac{1}{\LL^{1/2}-\LL^{-1/2}} \,\dim_!\DTS^{mot}_\mu \Bigr)  \\ &=& \Sym\Bigl( \frac{1}{\LL^{1/2}-\LL^{-1/2}}\, \DT^{mot}|_{\Lambda_\mu} \Bigr).
\end{eqnarray*}
We also use the notation $\DTS^{mot}_d=\DTS^{mot}|_{\Msp^{ss}_d}$ and $\DT^{mot}_d$ for the coefficient of $\DT^{mot}$ in front of $t^d$.
Let us give an alternative definition of the Donaldson--Thomas function $\DTS^{mot}_\mu$ using framed moduli spaces. Fix a $\mu$-generic stability condition $\zeta$. By applying the ``integration map'' $I^{ss}_\mu=\prod_{d\in \Lambda_\mu} I^{ss}_d$ to the identity (\ref{Hilbert_scheme_identity}) and  by using $\Sym(\LL^i a)=\sum_{n\ge 0} \LL^{ni}\Sym^n(a)$, we obtain
\begin{eqnarray*}
\lefteqn{\frac{1}{\LL-1}\Sym\Big(\sum_{0\not= d\in \Lambda_\mu} \frac{\LL^{fd}}{\LL^{1/2}-\LL^{-1/2}} \DTS^{mot}_d\Big) }\\
&=&\sum_{d\in \Lambda_\mu} \frac{\LL^{fd}}{\LL-1}p_{d\,!}(\ICS_{\Mst^{ss}_d}) \\
&=&  I^{ss}_\mu\Big( \sum_{d\in \Lambda_\mu} \frac{\LL^{fd}}{\LL-1} \unit_{\Mst^{ss}_d}\Big) \\
&=& I^{ss}_\mu(H)I^{ss}_\mu(\unit_{\Mst^{ss}_\mu})\\
&=& \Big(p_{!}\sum_{d\in \Lambda_\mu} \LL^{(d,d)/2}\tilde{\pi}_{d\, !} (\unit_{\Mst^{ss}_{f,d}})\Big)\Sym\Big(\frac{\DTS^{mot}_\mu }{\LL^{1/2}-\LL^{-1/2}}\Big) \\
&=& \Big(\pi_{!}\sum_{d\in \Lambda_\mu} \LL^{(d,d)/2}p_{f,d\,!} (\unit_{\Mst^{ss}_{f,d}})\Big)\Sym\Big(\frac{\DTS^{mot}_\mu }{\LL^{1/2}-\LL^{-1/2}}\Big) \\
&=& \frac{1}{\LL-1}\Big(\pi_{!}\sum_{d\in \Lambda_\mu} \LL^{fd/2}\ICS_{\Msp^{ss}_{f,d}}\Big)\Sym\Big(\frac{\DTS^{mot}_\mu }{\LL^{1/2}-\LL^{-1/2}}\Big), 
\end{eqnarray*}
where we applied equation (\ref{push_pull}) to the principal $(G_d\times\GG_m)$-bundle $X^{ss}_{f,d}\to \Mst^{ss}_{f,d}$ and to the principal $P(G_d\times \GG_m)=G_d$-bundle $X^{ss}_{f,d}\to \Msp^{ss}_{f,d}$ once more to compute $p_{f,d\,!}(\unit_{\Mst^{ss}_{f,d}})=\unit_{\Msp^{ss}_{f,d}}/(\LL-1)$.
Using the properties of $\Sym$ and $\frac{\LL^{fd}-1}{\LL^{1/2}-\LL^{-1/2}}=\LL^{1/2}[\PP^{fd-1}]$, we get the so-called DT/PT correspondence.
\begin{proposition}[DT/PT correspondence] \label{PT-DT_stack} For every quiver $Q$ and every $\mu$-generic stability condition $\zeta$ we get
\[ \pi_{!}\sum_{d\in \Lambda_\mu} \LL^{fd/2}\cdot\ICS_{\Msp^{ss}_{f,d}}= \Sym\Big(\sum_{0\not= d\in \Lambda_\mu} \LL^{1/2}[\PP^{fd-1}] \DTS^{mot}_d\Big)
\]
for all framing vectors $f\in \NN^{Q_0}$. 
\end{proposition}
If $f\in (2\NN)^{Q_0}$, we have $fd/2\in \NN$, and the map 
\[  (a_d)_{d\in \NN^{Q_0}} \longmapsto (\LL^{-fd/2}a_d)_{d\in \NN^{Q_0}} \]
is an isomorphism of the $\lambda$-ring  $\underline{\Ka}_0(\Var/\Msp^{ss}_\mu)[\LL^{-1/2},(\LL^r-1)^{-1} : r\ge 1]$ as $\Sym^n(\LL^{-fd/2}a_d)=\LL^{-nfd/2}\Sym^n(a_d)$ in this case. Applying this isomorphism to the DT/PT correspondence yields the alternative form.
\begin{corollary}[DT/PT correspondence, alternative form] For every quiver $Q$ and every $\mu$-generic stability condition $\zeta$ we get
\[ \pi_{!}(\ICS_{\Msp^{ss}_{f,\mu}})=\Sym\Big(\sum_{0\not= d\in \Lambda_\mu} [\PP^{fd-1}]_{vir} \DTS^{mot}_d\Big) \]
for all framing vectors $f\in (2\NN)^{Q_0}$ with $[\PP^{fd-1}]_{vir}=\int_{\PP^{fd-1}}\ICS_{\PP^{fd-1}}=\frac{\LL^{fd/2}-\LL^{-fd/2}}{\LL^{1/2}-\LL^{-1/2}}$. 
\end{corollary}
Notice that $\PP^{fd-1}$ is the fiber of $\pi_{d}$ over any geometric point of $\Msp^{st}_d$.  
\begin{corollary} \label{intconj3}
If $\zeta$ is generic, the motivic Donaldson--Thomas function $\DTS^{mot}$ is in the image of the map 
 \[ \underline{\Ka}_0(\Var/\Msp^{ss})[\LL^{-1/2},[\PP^N]^{-1}:r\ge 1] \longrightarrow \underline{\Ka}_0(\Var/\KK)[\LL^{-1/2}, (\LL^r-1)^{-1}: r\ge 1]\]
 and similarly for $\DT^{mot}$.
\end{corollary}

By applying the $\lambda$-ring homomorphism from $\underline{\Ka}_0(\Var/\Msp^{ss}_\mu)[\LL^{-1/2},(\LL^r-1)^{-1}: r\ge 1]$ to $\underline{\Ka}_0(\MHM(\Msp^{ss}_\mu))[\LL^{-/2},(\LL^r-1)^{-1}:r\ge 1]$, mentioned at the beginning of this section, to the previous result, we obtain the corresponding formula in $\underline{\Ka}_0(\MHM(\Msp^{ss}_\mu))[\LL^{-1/2},(\LL^r-1)^{-1}:r\ge 1]$.
\begin{corollary} \label{alternative_form} For every quiver $Q$ and every $\mu$-generic stability condition $\zeta$ we get
\[ \pi_{\ast}(\ICS_{\Msp^{ss}_{f,\mu}})=\pi_{!}(\ICS_{\Msp^{ss}_{f,\mu}})=\Sym\Big(\sum_{0\not= d\in \Lambda_\mu} [\PP^{fd-1}]_{vir} \DTS_d\Big) \]
for all framing vectors $f\in (2\NN)^{Q_0}$. 
\end{corollary}

\subsection{The integrality conjecture}

The so-called Integrality Conjecture plays a fundamental role in Donaldson--Thomas theory. A proof for quiver with potential has been sketched in \cite{KS2} in the Hodge theoretic context. A rigorous proof for quiver without potential and non-refined Donaldson--Thomas invariants can be found in \cite{Reineke3}. A relative version, saying that whenever the conjecture holds for one stability condition, it also holds for any other, has been given in \cite{JoyceDT} (see also \cite{Reineke3}). Our proof is different from the very complicated one given by Kontsevich and Soibelman. In fact, we reduce the general situation of quiver representations to a special situation for which the integrality conjecture has been proven by Efimov \cite{Efimov}. \\

As we have seen in Corollary \ref{intconj3}, the motivic Donaldson--Thomas invariants can be specialized to Euler characteristics producing rational numbers. The classical  integrality conjecture claims that these rational numbers are actually integers. We will prove a relative version of this in the motivic context. Let us assume  $\Char(\kk)=0$ for our ground field $\kk$. Unless otherwise stated, all schemes and stacks are defined over $\kk$.

\begin{theorem}[Integrality Conjecture, relative version]  \label{intconjsv}
Let $\zeta$ be a  $\mu$-generic stability condition and $x\in \Msp^{ss}_\mu$ a not necessarily closed point with residue field $\kk(x)$.  Then, there is a finite separable extension $\KK\supset \kk(x)$ depending on $x$ with induced morphism $i:\Spec \KK\to \Msp^{ss}_\mu$  such that $i^\ast\DTS^{mot}$ is in the image of the natural map 
\[ \Ka_0(\Var/\KK)[\LL^{-1/2}] \longrightarrow \Ka_0(\Var/\KK)[\LL^{-1/2}, (\LL^r-1)^{-1}: r\ge 1].\] 
\end{theorem} 
\begin{corollary}
 If $\zeta$ is $\mu$-generic and  $x\in \Msp^{ss}_\mu$ is a closed point with $\kk(x)=\overline{\kk(x)}$, then the ``value'' $\DTS^{mot}(x):=\DTS^{mot}|_{\Spec\kk(x)}$ of the Donaldson function $\DTS^{mot}$ at $x$ is in the image of 
 \[ \Ka_0(\Var/\kk(x))[\LL^{-1/2}] \longrightarrow \Ka_0(\Var/\kk(x))[\LL^{-1/2}, (\LL^r-1)^{-1}: r\ge 1].\] 
 The same applies to the value $\DTS^{mot}(y):=y^\ast\DTS^{mot}$ at any geometric point $y:\Spec\KK\to \Msp^{ss}_\mu$ of $\Msp^{ss}_\mu$.
\end{corollary}

\begin{proof}[Proof of the theorem]
Let $x\in \Msp^{ss}_d$ be a point of $\Msp^{ss}_\mu$ with residue field $\kk(x)$ and dimension vector $d$. As $R^{ss}_d\to \Msp^{ss}_d$ is of finite type and surjective on (geometric) points, we can certainly find a lift $\bar{x}\in R_d^{ss}$ with residue field $\kk(\bar{x})\supset \kk(x)$ being a finite extension. The point $\bar{x}$ corresponds to a semistable representation $V$ of $Q$ defined over $\kk(\bar{x})$ along with a choice of a basis of $V$ which is not important. By passing to a finite extension $\KK\supset \kk(\bar{x})$, we can assume that every stable Jordan--H\"older factor of $V$ remains stable under any base change. Indeed, the dimension of $V$ is finite and we cannot have an infinite chain of field extensions such that the number of Jordan--H\"older factors $E_k$ of $V$ strictly increases. Note that $\KK\supset \kk(x)$ is separable as $\Char(\kk)=0$. The associated  polystable representation for $V$ is $\bigoplus_{k=1}^sE_k^{a_k}$ with pairwise non-isomorphic stable representations $E_k$ of dimension vector $d^k=\dim E_k$ and multiplicity $a_k\in \NN\setminus \{0\}$. Hence,  $d=\sum_{k=1}^sa_kd^k$, and we write $E =(E_k)_{k=1}^s$ for the $s$-tuple of simple objects.\\

Changing the multiplicities, we get a family of polystable quiver representation on $\NN^s\times \Spec\KK$ with $\bigoplus_{k=1}^s E_k^{n_k}$ being the fiber over $n=(n_1,\ldots,n_k)\in \NN^s$. Let $\iota_E:\NN^s\times\Spec\KK \longrightarrow \Msp^{ss}_\mu$ be the associated (coarse) classifying map. By construction, the point corresponding to $(n_k)=(a_k)$ maps to $x$. \\
Note that $\underline{\Ka}_0(\Var/\NN^s\times\Spec\KK)[\LL^{-1/2}, (\LL^r-1)^{-1}: r\ge 1]$ can be identified with the ring \[\Ka_0(\Var/\KK)[\LL^{-1/2}, (\LL^r-1)^{-1}: r\ge 1][[t_1,\ldots,t_s]]\] of power series in $s$ variables. We will prove that $\iota_E^\ast \DTS^{mot}_\mu$ lies in the image of \[\Ka_0(\Var/\KK)[\LL^{-1/2}][[t_1,\ldots,t_s]] \longrightarrow \Ka_0(\Var/\KK)[\LL^{-1/2}, (\LL^r-1)^{-1}: r\ge 1][[t_1,\ldots,t_s]]\] 
which implies the theorem after restriction to the component indexed by $(n_k)=(a_k)$. Let us form the following fiber product: 
\[ \xymatrix @C=2cm { \Mst_E \ar[r]^{\tilde{\iota}_E} \ar[d]_{\tilde{p}} & \Mst^{ss}_\mu \ar[d]^p \\ \NN^s\times\Spec\KK \ar[r]^{\iota_E} & \Msp^{ss}_\mu } \] 
The stack $\Mst_E =\sqcup_{n\in \NN^s} \Mst_{E,n}$ can be seen as the stack of (semistable) representations defined over $\KK$ and having a decomposition series with factors in the collection $E =(E_k)_{k=1}^s$. We want to apply Lemma \ref{lambda_pull_back} to $N=\NN^s\times \Spec\KK$ and $M=\Msp^{ss}_\mu$. By our construction and the Krull--Schmidt theorem, $u_n$ is  a bijection between the points of the underlying schemes. Moreover, the local rings of $N\times_M M^{\times_\kk n}$ are $\KK$-algebras with a map to $\KK$ given by $\iota^\ast=\iota_E^\ast$. Thus, their residue field is $\KK$, and $u_n$ is a closed embedding inducing a bijection between geometric points. Hence, the Lemma applies. Since $p_!$ commutes with base change, we finally get 
\[ \tilde{p}_! \bigl( \tilde{\iota}_E^\ast \ICS^{mot}_{\Mst^{ss}_\mu} \bigr) = \Sym \Bigl( \frac{1}{\LL^{1/2}-\LL^{-1/2}} \iota_E^\ast\DTS^{mot}_\mu \Bigr). \]
Note that $\tilde{\iota}_E^\ast \ICS^{mot}_{\Mst^{ss}_\mu}$  restricted to $\Mst_{E,n}$ is just $\LL^{(d(n),d(n))/2}[\Mst_{E,n} \xrightarrow{\id} \Mst_{E,n}]$, where $d(n):=\sum_{k=1}^sn_kd^k$ is the dimension vector of $\bigoplus_{k=1}^s E^{n_k}_k$. \\
Let us introduce the ``Ext-quiver'' $Q_\xi $ of the collection $\xi=(d^k)_{k=1}^s$ of dimension vectors. Its vertex set is $\{1,\ldots,s\}$, and the number of arrows from $k$ to $l$ is given by $\delta_{kl}-(d^k,d^l)=\dim_\KK \Ext^1_{\KK Q\rep}(E_k,E_l)$. For a dimension vector $n\in \NN^s$ of $Q_\xi$, we denote with $R_{n}(Q_\xi)\cong \mathbb{A}_\KK^{\sum_{\alpha:k \to l}n_kn_l}$ the affine space parameterizing all representations of $Q_\xi$ on a fixed $\KK$-vector space of dimension $n$. Recall that $R_n(Q_\xi)/G_{n}$ with $G_n=\prod_{k=1}^s \Gl_\KK(n_k)$ is the stack of $n$-dimensional $\KK Q_\xi$-representations on any vector space of dimension vector $n$. \\   
As $(-,-)$ is symmetric by assumption on $\zeta$, the quiver $Q_\xi$ is symmetric, and we can apply the following result of Efimov to the quiver $Q_\xi$. 

\begin{theorem}[\cite{Efimov}, Theorem 1.1] 
Given any quiver $Q$ with vertex set $\{1,\dots,s\}$, we define for every $n\in \NN^s\setminus \{0\}$ the ``motivic'' Donaldson--Thomas invariant $\DT^{mot}(Q)_{n}\in \ZZ[\LL^{\pm 1/2}, (\LL^r-1)^{-1}: r\ge 1]$ of $Q$ with respect to the trivial stability condition $\theta=0$ by means of 
\[ \sum_{n\in \NN^s} \LL^{(n,n)/2}\frac{[R_{n}(Q)]}{[G_{n}]}\, t^n =: \Sym \Bigl( \frac{1}{\LL^{1/2}-\LL^{-1/2}} \sum_{n\in \NN^s\setminus\{0\}} \DT^{mot}(Q)_{n}t^n \Bigr), \]
where we might think of $\LL^{1/2}$ as a formal variable. 
If the quiver $Q$ is symmetric, the invariant $\DT^{mot}(Q)_n$ is contained in the Laurent subring $\ZZ[\LL^{\pm 1/2}]$ of $\ZZ[\LL^{\pm 1/2}, (\LL^r-1)^{-1}: r\ge 1]$.
\end{theorem}

When we apply Efimov's Theorem to $Q_\xi$ and specialize $\LL$ to $[\AA_\KK^1]$, we use the notation $(-,-)_{Q_\xi}, R_{n}(Q_\xi)$ and $\DT^{mot}(Q_\xi):=\sum_{n\in \NN^s\setminus\{0\}} \DT^{mot}(Q_\xi)_nt^n$ to distinguish the objects from their counterparts for $Q$. Theorem \ref{intconjsv} is then a direct consequence of the following result.  \end{proof}

\begin{proposition} \label{localDT}
Let us denote with $\DT^{mot}(Q_\xi)|_{\LL^{1/2}\mapsto \LL^{-1/2}}$ the series in \\ $\ZZ[\LL^{\pm 1/2}][[t_1,\ldots,t_s]]$ obtained by the indicated substitution. If $\zeta$ is $\mu$-generic, then \[\DT^{mot}(Q_\xi)|_{\LL^{1/2}\mapsto \LL^{-1/2}}=\iota_E^\ast \DTS^{mot}_\mu.\] 
In particular, $\iota_E^\ast\DTS^{mot}_\mu$ is an element of the subring $\ZZ[\LL^{\pm 1/2}][[t_1,\ldots,t_s]]$ which also embeds into  the subring $\Ka_0(\Var/\KK)[\LL^{-1/2}][[t_1,\ldots,t_s]]$ of $\Ka_0(\Var/\KK)[\LL^{-1/2}, (\LL^r-1)^{-1}: r\ge 1][[t_1,\ldots,t_s]]$.
\end{proposition}

\begin{remark} \rm
The substitution $\LL^{1/2}\mapsto \LL^{-1/2}$ has an intrinsic meaning. For any base $\BB$ there is a duality operation on $\underline{\Ka}_0(\Var/\BB)[\LL^{-1/2}, (\LL^r-1)^{-1}: r\ge 1]$ which can be seen as a motivic version of (relative) Poincar\'{e} duality. See \cite{Bittner04}, section 6 for more details on this.
\end{remark}

\begin{proof} As the substitution $\LL^{1/2}\mapsto \LL^{-1/2}$ is compatible with the $\lambda$-ring structure of $\ZZ[\LL^{\pm 1/2}, (\LL^r-1)^{-1}: r\ge 1][[t_1,\ldots,t_s]]$, which contains $\ZZ[\LL^{\pm 1/2}][[t_1,\ldots,t_s]]$ as a $\lambda$-subring, it suffices to show the identity 
\begin{equation} \label{nilpotent} \Bigl(\sum_{n\in \NN^s} \LL^{(n,n)_{Q_\xi}/2} \frac{[R_{n}(Q_\xi)]}{[G_n]}t^n \Bigr)\Big|_{\LL^{1/2} \to \LL^{-1/2}} \cdot \Bigl(\sum_{m\in \NN^s} \LL^{(d(m),d(m))/2} [\Mst_{E,m}]t^m\Bigr) = 1 
\end{equation}
in $\Ka_0(\Var/\KK)[\LL^{-1/2}, (\LL^r-1)^{-1}: r\ge 1][[t_1,\ldots,t_s]]$. Indeed, the factor on the left hand side is by definition 
\[\Sym\Bigl(\frac{\DT^{mot}(Q_\xi)}{\LL^{1/2}-\LL^{-1/2}}\Bigr)\Bigr|_{\LL^{1/2}\mapsto \LL^{-1/2}} = \Sym\Bigl(- \frac{\DT^{mot}(Q_\xi)|_{\LL^{1/2}\mapsto \LL^{-1/2}}}{\LL^{1/2}-\LL^{-1/2}}\Bigr).\]  
On the other hand, the factor on the right hand side is nothing else than 
\[ \tilde{p}_! (\tilde{\iota}_E^\ast \ICS_{\Mst^{ss}_\mu}) = \Sym\Bigl( \frac{ \iota_E^\ast \DTS^{mot}_\mu}{\LL^{1/2}-\LL^{-1/2}}\Bigr). \]
Consider the following two motivic functions on $\Mst^{ss}_{\mu,\KK}:=\Mst^{ss}_\mu\times_\kk \Spec\KK$. 
\[ f:=\sum_{n\in \NN^s}(-1)^{|n|}\LL^{\sum_{k=1}^s{n_k \choose 2}}[ \Spec \KK /G_n \rightarrow \Mst^{ss}_{\mu,\KK}] \quad\mbox{and}\quad g:=[\Mst_E \rightarrow \Mst^{ss}_{\mu,\KK}],\]
where for $n\in \NN^s$ the quotient stack $\Spec \KK/G_n$ maps to the object $\bigoplus_{k=1}^s E_k^{n_k}$ of dimension vector $d(n)$ and its automorphism group. In particular, the morphisms used to define $f$ and $g$ correspond to closed substacks of $\Mst^{ss}_{\mu,\KK}$. We compute the convolution product $f\ast g$ by means 
of the following diagram
\[ \xymatrix { \mathcal{Z}_{d(n),d(m)} \ar@{^{(}->}[r] \ar[d] & \mathfrak{Exact}_{d(n),d(m),\KK} \ar[d]_{\pi_1\times \pi_3} \ar[r]^{\pi_2} & \Mst^{ss}_{d(n)+d(m),\KK} \\ \Spec \KK/G_n \times_\KK \Mst_{E,d(m)} \ar@{^{(}->}[r] & \Mst^{ss}_{d(n),\KK} \times_\KK \Mst^{ss}_{d(m),\KK},  } \] 
with the  square being cartesian and $\mathfrak{Exact}_{d(n),d(m),\KK}$ denoting the stack of short exact sequences  in $\KK Q\rep$ with prescribed dimensions for the first and third object in the sequence. The morphisms $\pi_1,\pi_2$ and $\pi_3$ map a sequence to the the corresponding entries. Since $\pi_2$ is representable, $\mathcal{Z}_{d(n),d(m)} \longrightarrow \Mst^{ss}_{d(n)+d(m),\KK}$ is representable, too. In fact, $\mathcal{Z}_{d(n),d(m)}$ maps to the substack of $\Mst_E$ parameterizing representations $F$ that are extensions of a representation  with dimension vector $d(m)$ and Jordan--H\"older factors among the $(E_k)_{k=1}^s$ by the polystable representation $\bigoplus_{k=1}^s E_k^{n_k}$. In particular, the  Jordan--H\"older factors of $F$ are also among the $(E_k)_{k=1}^s$, and $\bigoplus_{k=1}^s E_k^{n_k}$ must embed into the socle $\bigoplus_{k=1}^s E_k^{N_k}$ of $F$ for certain integers $N_k$ depending on $F$. The space of such embeddings, that is, the fiber of the map $\mathcal{Z}_{d(n),d(m)} \longrightarrow \Mst^{ss}_{d(n)
+d(m),\KK}$ over $F$, is given by the product of finite Grassmannians $\prod_{k=1}^s\Gr_{n_k}^{N_k}$ over $\KK$. Hence, the convolution product $f\ast g$ restricted to $F \in \Mst^{ss}_{d(n)+d(m),\KK}$ is
\[(f\ast g)|_{\Spec\KK(F)}=\sum_{0\leq n_k\leq N_k}\prod_{k=1}^s(-1)^{n_k}\LL^{{n_k}\choose 2}\left[{N_k\atop n_k}\right],\] 
in $\Ka_0(\Var/\KK(F))$ since the $\LL$-binomial coefficient $\left[{N_k\atop n_k}\right]$ are the  motives of the Grassmannians $\Gr^{N_k}_{n_k}$. This identity does not only hold pointwise. For any dimension vector $l\in \NN^s$ let $R^{E}_{d(l)}\subset R_{d(l),\KK}:=R_{d(l)}\times_\kk\Spec\KK$ denote the atlas of $\Mst_{E,l}$. It is a closed subset of $R_{d(l),\KK}$ containing only finitely many closed closed orbits for the group $G_{d(l),\KK}=G_{d(l)}\times_\kk\Spec\KK$. The socle of the universal (trivialized) family $\FF$ on $R^E_{d(l)}$ is the image of the monomorphism
\[
\bigoplus_{k=1}^s E_k\otimes_\KK \mathcal{H}om(E_k, \FF) \longrightarrow \FF. 
\]
The family $\mathcal{H}om(E_k, \FF)$ of linear spaces is a vector bundle when restricted to a  stratification of $R^E_{d(l)}$. The $G_{d(l),\KK}$-invariant strata $S_N$ indexed by $r\ge 1^s$ contain the points $M\in R^E_{d(l)}$ with $\dim_\KK \mathcal{H}om(E_k, \FF)|_M=N_k$ for all $1\le k\le s$. For $n+m=l$, let is form the fiber product
\[ \xymatrix { Z_{d(n),d(m),N} \ar[d] \ar[r]^\theta & S_N \ar[d] \\ \mathcal{Z}_{d(n),d(m)} \ar[r] & \Mst^{ss}_{d(l),\KK}. } \]
The map $\theta$ is just the product of the relative Grassmannians of the vector bundles $\mathcal{H}om(E_k,\FF)$ on $S_N$. It  is a Zariski locally trivial $\prod_{k=1}^s\Gr_{n_k}^{N_k}$-fibration. The vertical maps are principal $G_{d(l),\KK}$-bundles over their  image $Z_{d(n),d(m),N}/G_{d(l),\KK}$ and $S_N/G_{d(l),\KK}$ respectively. The images are locally closed substacks of  $\mathcal{Z}_{d(n),d(m)}$ and $\Mst^{ss}_{d(l),\KK}$ respectively. Summing up over all $m,n,r\ge 1^s$ with fixed $n+m=l$ and using equation (\ref{push_pull}), we get 
\[ (f\ast g)|_{\Mst^{ss}_{d(l),\KK}}= \sum_{r\ge 1^s} \Big(\sum_{0\leq n_k\leq N_k}\prod_{k=1}^s(-1)^{n_k}\LL^{{n_k}\choose 2}\left[{N_k\atop n_k}\right]\Big)[S_N/G_{d(l),\KK}\hookrightarrow \Mst^{ss}_{d(l),\KK}] \] 
as we want. Note that the outer sum is finite as $S_N\not=\emptyset$ for only finitely many $N$. A standard identity for $\LL$-binomial coefficients  shows that the term in the big brackets vanishes as soon as $N\not=0$. The case $N=0$ can only give a nonzero contribution if $l=d(l)=0$ as every nontrivial representation has a nontrivial socle. One shows easily $(f\ast g)|_{\Mst^{ss}_{0,\KK}}=1$, and the formula $f\ast g= 1$ is proven. 
Using Lemma \ref{integration_map}, we get the identity $1=I(f\ast g)=I(f)\cdot I(g)$ of motivic functions on $\Msp^{ss}_\mu\times_\kk\Spec\KK$ which are actually supported on the closed subscheme $\NN^s\times \Spec\KK \hookrightarrow \Msp^{ss}_\mu\times_\kk\Spec\KK$ via the embedding induced by $\iota_E$. Using $[R_{n}(Q_\xi)]=\LL^{-(n,n)_{Q_\xi}+\sum_{k=1}^sn_k^2}=\LL^{-(d(n),d(n))+\sum_{k=1}^s n_k^2}$, a simple computation shows that $I(f)$ is the first factor in equation (\ref{nilpotent}) while the second is obviously $I(g)$. 
\end{proof}
\begin{corollary} \label{localDT2}
Let $V=\bigoplus_{k=1}^s E_k^{m_k}$ be a polystable $\KK Q$-representation corresponding to a $\KK$-point $y:\Spec\KK\to  \Msp^{ss}_\mu$. Assume that the stable representations $E_k$ remain stable under base change. As before, $Q_\xi$ denotes the $\Ext^1$-quiver of the collection $(E_k)_{k=1}^s$ of stable objects. Let $\DT^{mot}(Q_\xi)^{nilp}_m:=\DTS^{mot}(Q_\xi)(0_m)$ be the ``value'' of $\DTS^{mot}(Q_\xi)$ (with respect to the trivial stability condition) at the ``origin'' in $\Msp(Q_\xi)_m$ corresponding to the zero-representation $0_m$ of dimension $m=(m_k)_{k=1}^s$. If $\zeta$ is $\mu$-generic, then $\DTS^{mot}(y):=y^\ast\DTS^{mot}=\DT^{mot}(Q_\xi)_m^{nilp}$ for the value  of $\DTS^{mot}$ at $y:\Spec\KK\to \Msp^{ss}_\mu$.
\end{corollary}
\begin{proof}
The zero-representation $0_m$ of dimension m is the semisimple $Q_\xi$-representation $\bigoplus_{k=1}^s S_k^{m_k}$, where $S_k$ denotes the 1-dimensional zero-representation of $\KK Q_\xi$ at vertex $k$. We simply apply Proposition \ref{localDT} to the category  $\KK Q_\xi\rep$ and the collection $(S_k)_{k=1}^s$. One should also take into account that the local $\Ext^1$-quiver of this collection is $Q_\xi$ again. Thus, $\DTS^{mot}(Q_\xi)(0_m) =\DT^{mot}(Q_\xi)|_{\LL^{1/2}\mapsto \LL^{-1/2}}=y^\ast \DTS^{mot}$. 
\end{proof}

\begin{corollary} \label{integrality_stratification}
 If  $\zeta$ is $\mu$-generic, there is a stratification of $\Msp^{ss}_\mu$ into connected strata $S_\kappa$ and there are \'{e}tale covers $j_\kappa:\tilde{S}_\kappa \to S_\kappa$ of the strata such that $\DTS^{mot}|_{\tilde{S}_\kappa}:=j_\kappa^\ast \DTS^{mot}$ is in the image of 
 \[ \Ka_0(\Var/\tilde{S}_\kappa)[\LL^{-1/2}] \longrightarrow \Ka_0(\Var/\tilde{S}_\kappa)[\LL^{-1/2}, (\LL^r-1)^{-1}: r\ge 1].\]
\end{corollary}
\begin{proof}
 It is enough to construct such a stratification on each scheme $\Msp^{ss}_d$ with $d\in \Lambda_\mu$. In order to prove the corollary, it suffices to construct an \'{e}tale neighborhood of  the generic point $x$ of $\Msp^{ss}_{d}$ and to show the absence of denominators on this neighborhood. If that has been done, we can restrict ourselves to the closed complement $Z$ of the  open image of this neighborhood and proceed with the generic points of the irreducible components of $Z$. Continuing this way, we get lots of \'{e}tale neighborhoods $\tilde{S}_\kappa$ inside closed subvarieties, and  $S_\kappa$ will denote their locally closed image in $\Msp^{ss}_\mu$. \\
To show the absence of denominators on an \'{e}tale neighborhood of the generic point $x$ of an irreducible subscheme inside $\Msp^{ss}_d$, we can use the alternative definition of $\Ka_0(\Var/\Spec B)$ given in Remark \ref{reduction_to_affine_case}. Write $\Spec A$ for a  Zariski neighborhood of $x$ and choose a finite separable extension $\KK\supset \kk(x)$ as in Theorem \ref{intconjsv}. Denote with $B$ the normalization of $A\subset \kk(x)$ inside $\KK$. Of course, $\KK=\Quot(B)$ is the quotient field of $B$. Replacing $\Spec A$ with an affine open subscheme, we can assume that $\Spec B\to \Spec A$ is an \'{e}tale cover, i.e.\ $\Spec B$ an \'{e}tale neighborhood of the generic point $x$. To prove the absence of denominators on $\Spec B$, or an open affine subscheme of $\Spec B$, we  have to show the following for arbitrary $r\ge 1$ and arbitrary $f\in \Ka_0(\Var/\Spec B)$: If there is an element $g \in \Ka_0(\Var/\Quot(B))$ given by linear combinations of finitely generated $\Quot(B)$-algebras such that  $f\otimes_B \Quot(B)=g(\LL^r-1)=g\otimes_{\Quot(B)} \Quot(B)[x_1,\ldots,x_r]- g$, then one can find elements $b\in B$ and $\tilde{g}\in\Ka_0(\Var/\Spec B_b)$ given by linear combinations of finitely generated $B_b$-algebras  such that $f\otimes_B B_b=\tilde{g}(\LL^r-1)=\tilde{g}\otimes_{B_{b'}} B_{b'}[x_1,\ldots,x_r]- \tilde{g}$. In such a situation, we may replace the open neighborhood of $x$ with $\Spec B_b$ and cancel a denominator of the form $\LL^r-1$. \\
As any finite set of finitely generated $\Quot(B)$ algebras is already defined over $B_{b'}$ for some $b'$, we can certainly ``lift'' $g$ to some $g'$. It remains to show that $f\otimes_B B_b=g'\otimes_{B_{b'}} B_{b'}[x_1,\ldots,x_r]- g'$. Over $\Quot(B)$ this is true due to the existence of a finite chain of relations presented in Remark \ref{reduction_to_affine_case}. But each of these relations does also lift to a relation over $B_b$ for some sufficiently ``large'' $b\in B\subset B_{b'}$. Then $\tilde{g}:=g'\otimes_{B_{b'}} B_{b}$ does what we want. 
\end{proof}
The following result is also a consequence of Theorem \ref{DT=IC}, but the previous corollary allows a more direct proof without any knowledge about mixed Hodge modules. 
\begin{corollary}[Integrality Conjecture, classical version] If $\zeta$ is $\mu$-generic, the motivic Donaldson--Thomas function $\DTS^{mot}_\mu$ has a realization in integer valued constructible functions on $\Msp^{ss}_\mu$. In particular, the Euler characteristic of $\DT_d^{mot}$ is an integer for all $d\in \Lambda^\mu$.
\end{corollary}
We can even refine the last statement of the corollary to motives.
\begin{corollary}[Integrality Conjecture, absolute version] 
For a $\mu$-generic stability condition $\zeta$ and arbitrary dimension vector $d\in \Lambda_\mu$, the Donaldson--Thomas invariant $\DT^{mot}_d$ is in the image of the natural map 
\[ \Ka_0(\Var/\kk)[\LL^{-1/2}] \longrightarrow \Ka_0(\Var/\kk)[\LL^{-1/2}, (\LL^r-1)^{-1}: r\ge 1].\]
\end{corollary}
\begin{proof}
 Unfortunately, the previous statement holds only for an ``\'{e}tale stratification''. If it were true for a Zariski stratification, i.e.\ $\tilde{S}_\kappa=S_\kappa$, we could  just integrate the Donaldson--Thomas function over $\Msp^{ss}_d$. As we do not have such a result, we need to argue in a different way. By applying  Lemma \ref{integration_map} and $\dim_!$ to the formula of Theorem 5.1 in \cite{Reineke_HN}, one shows easily that $\DT^{mot}_d$ is an element of the subring $\ZZ[\LL^{\pm 1/2}][(\LL^r-1)^{-1}:r\ge 1]$ of $\Ka_0(\Var/\kk)[\LL^{-1/2}, (\LL^r-1)^{-1}:r\ge 1]$ with coefficients being independent of the ground field. In particular, it can be identified with the weight ``polynomial'' of its Hodge realization $\IC_c(\overline{\Msp^{st}_d},\QQ)$ due to Theorem \ref{DT=IC}. Therefore, $\DT^{mot}_d$ is indeed in $\ZZ[\LL^{\pm 1/2}]\subset \Ka_0(\Var/\kk)[\LL^{-1/2}]$.   
\end{proof}

As we have seen, it would be nice to improve Theorem \ref{intconjsv} in such a way that integrality holds already for $\DTS^{mot}(x)=\DTS^{mot}|_{\Spec\kk(x)}$ at any point $x\in \Msp^{ss}_\mu$. In this case, we can even prove integrality of $\DTS^{mot}_\mu$ following the arguments of Corollary \ref{integrality_stratification} which of course implies the result for points. However, we are rather skeptical  that such an improvement exists in the (naive) motivic world, due to the following argument. The map $R^{st}_d\to \Msp^{st}_d$ is in general just an \'{e}tale locally trivial principal $PG_d=G_d/\GG_m$-bundle, and $PG_d$ is not special if $\gcd(d_i:i\in Q_0)\not=1$. Hence, its fiber $F$ at the generic point $x\in \Msp^{st}_d$ is a twisted form of $PG_d$. If relative integrality holds in the stronger form, we get a motive $M:=\LL^{\frac{1-(d,d)}{2}}\DTS^{mot}(x)\in \Ka_0(\Var/\kk(x))[\LL^{-1/2}]$ with $[F]=[PG_d]M$. After base change $M$ becomes $1$ which does, however, not imply $M=1$. Similarly, working with the Hilbert--Chow morphism, we get $[Q]=[\PP^{fd-1}]M$ for all (even) $f\in \NN^{Q_0}$, where $Q$ is a twisted form of $\PP^{fd-1}$. In general, (naive) motives of twisted forms behave very different. Over finite fields, the numbers of $\mathbbm{F}_p$-rational points, which is a motivic invariant, do not coincide. \\
Due to the relative hard Lefschetz theorem, the Hodge realization cannot distinguish between \'{e}tale locally trivial $\PP^{fd-1}$-fibrations and the trivial one. This was definitely used to prove the integrality of $\DTS_\mu$.

\bibliographystyle{plain}
\bibliography{Literatur}

\vfill
\textsc{\small S. Meinhardt: Fachbereich C, Bergische Universit\"at Wuppertal, Gau{\ss}stra{\ss}e 20, 42119 Wuppertal, Germany}\\
\textit{\small E-mail address:} \texttt{\small meinhardt@uni-wuppertal.de}\\
\\

\textsc{\small M. Reineke: Fachbereich C, Bergische Universit\"at Wuppertal, Gau{\ss}stra{\ss}e 20, 42119 Wuppertal, Germany}\\
\textit{\small E-mail address:} \texttt{\small mreineke@uni-wuppertal.de}\\

\end{document}